\documentclass{article}
\usepackage[utf8]{inputenc}

\usepackage{algorithm}
\usepackage{algorithmicx}
\usepackage{algpseudocode}
\usepackage{amsmath}
\usepackage{amssymb} 
\usepackage{amsfonts}
\usepackage{amsthm}
\usepackage{arydshln}
\usepackage{authblk}

\usepackage{bbold}
\usepackage{bm}
\usepackage{booktabs}
\usepackage{bookmark}
\usepackage{calc}              
\usepackage{cleveref}
\usepackage{enumitem}
\usepackage{extarrows}
\usepackage{float}
\usepackage{geometry}
\usepackage{graphicx}          
\usepackage{indentfirst}
\usepackage{latexsym}
\usepackage{mathtools}
\usepackage{multirow}          
\usepackage{subcaption}
\usepackage{tikz}              
\usepackage{xcolor}            
\definecolor{Violet}{named}{violet}

\title{CEM-GMsFEM for Poisson equations in heterogeneous perforated domains}

\date{}

\newtheorem{lemma}{Lemma}
\newtheorem{theorem}{Theorem}

\usepackage{authblk}

\hypersetup{colorlinks=true,
            linkcolor=blue,
            anchorcolor=blue,
            citecolor=blue}

\begin{document}
\author[1]{Wei Xie}
\author[2]{Yin Yang}
\author[3]{Eric Chung}
\author[4]{Yunqing Huang}
\affil[1]{National Center for Applied Mathematics in Hunan, Xiangtan University, Xiangtan 411105, Hunan, China. \texttt{xiew@smail.xtu.edu.cn}}
\affil[2]{School of Mathematics and Computational Science, Xiangtan University, Xiangtan 411105, Hunan, China \texttt{yangyinxtu@xtu.edu.cn}}
\affil[3]{Department of Mathematics, The Chinese University of Hong Kong, Shatin, Hong Kong SAR, \texttt{tschung@math.cuhk.edu.hk}}
\affil[4]{School of Mathematics and Computational Science, Xiangtan University, Xiangtan 411105, Hunan, China \texttt{huangyq@xtu.edu.cn}}

\maketitle

\begin{abstract}
In this paper, we propose a novel multiscale model reduction strategy tailored to address the Poisson equation within heterogeneous perforated domains. 
The numerical simulation of this intricate problem is impeded by its multiscale characteristics, necessitating an exceptionally fine mesh to adequately capture all relevant details. 
To overcome the challenges inherent in the multiscale nature of the perforations, we introduce a coarse space constructed using the Constraint Energy Minimizing Generalized Multiscale Finite Element Method (CEM-GMsFEM). 
This involves constructing basis functions through a sequence of local energy minimization problems over eigenspaces containing localized information pertaining to the heterogeneities. 
Through our analysis, we demonstrate that the oversampling layers depend on the local eigenvalues, thereby implicating the local geometry as well. 
Additionally, we provide numerical examples to illustrate the efficacy of the proposed scheme.
\end{abstract}

\section{Introduction}
Among multiscale problems, the problems in perforated domains are of great interest to many applications. These problems are characterized by processes taking place in domains with multiple scales, such as the presence of inclusions within the domain. Applications of perforated domain problems span a wide range, including fluid flow in porous media, diffusion in perforated domains, and mechanical processes in hollow materials, among others. Typically, the interaction between physical processes and heterogeneous media gives rise to challenges in perforated domains.

Numerous model reduction techniques have been developed in the existing literature to enhance computational efficiency for problems involving perforations. For instance, numerical upscaling methods \cite{hillairet2018homogenization, hornung1997homogenization, yosifian1997some} derive upscaled models and solve resulting problems globally on coarse grids, significantly reducing computational costs. Additionally, various multiscale methods have been proposed for simulating such problems. Multiscale finite element methods (MsFEM) \cite{hou1997multiscale, jiang2017reduced} of Crouzeix–Raviart type have been developed for elliptic problems \cite{le2014msfem} and Stokes flows \cite{muljadi2015nonconforming}. The Heterogeneous Multiscale Method (HMM) \cite{henning2009heterogeneous, weinan2003heterognous} discretizes elliptic problems with perforations on coarse grids. Moreover, generalized finite element methods based on the idea of localized orthogonal decomposition (LOD) \cite{maalqvist2014localization, henning2014localized} have been proposed for elliptic problems \cite{brown2016multiscale}.

Recently, the Generalized Multiscale Finite Element Method (GMsFEM) \cite{efendiev2013generalized, chung2016adaptive} has emerged as a promising framework for systematically enriching coarse spaces and incorporating fine-scale information in the construction of these spaces. This approach offers a convincing strategy for solving problems posed in heterogeneous perforated domains, characterized by multiscale features in their solutions and necessitating sophisticated enrichment techniques. The fundamental concept of GMsFEM involves employing local snapshots to approximate the fine-scale solution space, followed by the identification of local multiscale spaces through carefully selected local spectral problems defined within the snapshot spaces. These spectral problems provide a systematic approach to identifying dominant modes in the snapshot spaces, which are then selected to form the local multiscale spaces. By judiciously choosing the snapshot space and the spectral problem, GMsFEM requires only a few basis functions per coarse region to achieve solutions with excellent accuracy. 
Previous works have successfully applied GMsFEM to various problems, including Darcy's flow model, Stokes equations, coupled flow and transport, and others in heterogeneous domains \cite{chung2017online, chung2016mixed, chung2017conservative, chung2018multiscale, tyrylgin2021multiscale}.

\begin{figure}[htbp]
\centering
\begin{tikzpicture}
\draw[fill=teal] (0, 0) rectangle (6, 6);
\draw[fill=white] (0.1, 5.1) rectangle (0.9, 5.9);
\draw[fill=white] (5, 4.5) ellipse (0.4 and 0.2);
\draw[fill=white] (2.5, 1) ellipse (0.2 and 0.4);
\draw[fill=white] (2, 4.5) circle (0.5);
\draw[fill=white] (3.5, 5) circle (0.3);
\draw[fill=white] (5, 2) circle (0.4);
\draw[fill=white] (4, 1) circle (0.4);
\draw[fill=white] (3, 3) circle (0.4);
\draw[fill=white] (3.5, 3) circle (0.04);
\draw[fill=white] (2.5, 3) circle (0.04);
\draw[fill=white] (3, 3.5) circle (0.04);
\draw[fill=white] (3, 2.5) circle (0.04);
\draw[fill=white] (3.27, 3.4) circle (0.04);
\draw[fill=white] (3.27, 2.6) circle (0.04);
\draw[fill=white] (2.72, 3.4) circle (0.04);
\draw[fill=white] (2.72, 2.6) circle (0.04);
\draw[fill=white] (1, 2) circle (0.4);
\draw[fill=white] (1.5, 2) circle (0.04);
\draw[fill=white] (0.5, 2) circle (0.04);
\draw[fill=white] (1, 1.5) circle (0.04);
\draw[fill=white] (1, 2.5) circle (0.04);
\draw[fill=white] (1.27, 2.4) circle (0.04);
\draw[fill=white] (1.27, 1.6) circle (0.04);
\draw[fill=white] (0.72, 2.4) circle (0.04);
\draw[fill=white] (0.72, 1.6) circle (0.04);
\node at (0.5,5.5)   {$\Omega^{\epsilon}$};
\end{tikzpicture}
\caption{Illustration of a perforated domain.}
\label{pic:perforateddomain_example}
\end{figure}
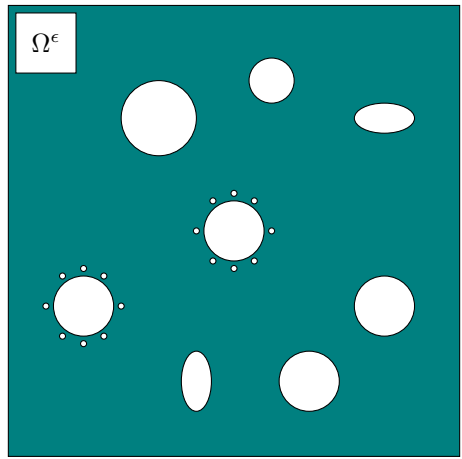

In this work, we present a multiscale model reduction technique for Poisson equation in perforated domain (see \autoref{pic:perforateddomain_example}). Our idea is motivated by the recently-developed Constraint Constraint Energy Minimizing Generalized Multiscale Finite Element Method (CEM-GMsFEM) \cite{chung2018constraint, chung2018fast, mmrbook}. 
This method has been applied successfully in dealing with many problems, e.g., poroelasticity problems \cite{fu2020constraint, fu2019generalized}, quasi-gas problem \cite{chetverushkin2021computational}, stokes problems in perforated domains \cite{chung2021convergence}, and others \cite{cheung2021explicit, ye2023constraint, wang2024multiscale}. 
CEM-GMsFEM is based on the framework of GMsFEM to design multiscale basis functions such that the convergence of the method is independent of the contrast from the heterogeneities; and the error linearly decreases with respect to coarse mesh size if oversampling parameter is appropriately chosen. 
The multiscale basis functions in CEM-GMsFEM consists of two steps. One needs to first construct auxiliary basis functions by solving local spectral problems. Then, for each auxiliary basis function, one can construct a multiscale basis via energy minimization problems on an oversampling subdomains. 
We propose two versions, the first one is based on solving constraint energy minimization problems and the second one is the relax version by solving unconstrained energy minimization problems. In the following, we use constraint version for the former one while relaxed version to distinguish these two method.
Consider the oversampling will lead extra computation, we will show the suitable oversampling layers depends on the eigenvalues of the local domain.

The paper is organized as follows. In Section \ref{sec:preliminaries}, we present the model problem and the fine-scale approximation. In Section \ref{sec:cembasis}, we introduce auxiliary space and the construction of multiscale basis functions, both constraint version and relaxed version. Section \ref{sec:analysis} is devoted to analysis of the approach in Section \ref{sec:cembasis} and show the decay is depends on the eigenvalues of the local domain. We will give some numerical results in Section \ref{sec:numericalresults} to show the efficiency of our proposed method. Conclusions will be included in Section \ref{sec:conclusions}.

\section{Preliminaries} \label{sec:preliminaries}
In this section, we state the Poisson equation in heterogeneous perforated domains and introduce some notations. 
Let $\Omega \subset \mathbb{R}^d (d=2,3)$ be a bounded domain. 
We set $\Omega^{\epsilon} \subset \Omega$ be a perforated domain and $\mathcal{B}^{\epsilon}$ be the perforations with $\Omega$, that is $\Omega^{\epsilon} = \Omega \setminus \mathcal{B}^{\epsilon}$. 
The set of perforations $\mathcal{B}^{\epsilon}$ is assumed to be a union of connected circular disks, see \autoref{pic:perforateddomain_example} for illustration. 
Each of these disks is of diameter of order $0< \epsilon \ll \mathrm{diam}(\Omega)$. 
In perforated domain, we consider 

\begin{equation}
\begin{cases}
\begin{aligned} 
- \Delta u &= f,
& \textnormal{in} ~\Omega^{\epsilon},\\
\frac{\partial u}{\partial \boldsymbol{n}} &= 0 ,
& \textnormal{on}~ 
\partial \Omega^{\epsilon} \cap \partial 
\mathcal{B}^{\epsilon}, \\
u &= 0, & \textnormal{on}~ 
\partial \Omega^{\epsilon} \cap \partial \Omega. \\
\end{aligned}
\end{cases}
\label{eq:pde_perforated}
\end{equation}
where $\boldsymbol{n}$ stands for outward unit normal vectors to $\partial \Omega^{\epsilon}.$
To define the weak formulation of \eqref{eq:pde_perforated}, we introduce the space
$V \coloneqq \{v\in H^1(\Omega^{\epsilon}) : 
v|_{\partial \Omega^{\epsilon} \cap \partial \Omega}=0\}.$
The variational formulation of \eqref{eq:pde_perforated} is: 
Find $u \in V$, such that

\begin{equation}
a(u, v) = (f, v), \quad \forall v \in V,
\label{eq:perforated_variation}
\end{equation}
where

$$
a(u,v) = \int_{\Omega^{\epsilon}} \nabla u \nabla v, \quad 
(f, v) = \int_{\Omega^{\epsilon}} f v.
$$

Let $\mathcal{T}^h$ be a triangulation of perforated domain $\Omega^{\epsilon}$. In here, $h$ is the triangle mesh size, we assume $\mathcal{T}^h$ is sufficiently fine to fully resolve the small-scale information of the domain $\Omega^{\epsilon}.$ We can define a finite space $V_h$ via the linear basis functions on $\mathcal{T}^h.$ The continuous Galerkin (CG) formulation of \eqref{eq:pde_perforated} is : Find $u_h \in V_h$, such that 

\begin{equation}
a(u_h, v_h) = (f, v_h), \quad \forall v_h \in V_h.
\label{eq:perforated_galerkin}
\end{equation}
Let $\varphi_i$ to denote the linear basis functions in $\mathcal{T}^h.$
Then, $V_h = \{\varphi_i\}_{i=1}^N$, where $N$ is the interior nodes within the mesh $\mathcal{T}^h$ and boundary nodes intersecting $\mathcal{B}^{\epsilon}.$
We do not need any hat functions on the boudary of $\Omega$, because the functions in $V_h$ must vanish here. 
Now, using this basis, we can conclude the finite element method \eqref{eq:perforated_galerkin} is equivalent to 

\begin{equation}
a(u_h, \varphi_i) = (f, \varphi_i), \quad \forall \varphi_i \in V_h.
\label{eq:perforated_fem_varphi}
\end{equation}
the equivalent matrix form of \eqref{eq:perforated_fem_varphi} is 

\begin{equation}    
A u_h = F
\label{eq:perforated_fem_matrix}
\end{equation}
where 

$$
A_{ij} = a(\varphi_j, \varphi_i), \quad 
F_i = (f, \varphi_i).
$$

In this paper, we will use a multiscale model reduction technique that construct a multiscale space $V_{\textnormal{ms}}$ to reduce the cost compare to \eqref{eq:perforated_fem_matrix}.
We have to point out that the multiscale space $V_{\textnormal{ms}}$ is a subset of $V$, that is $V_{\textnormal{ms}} \subset V.$
The multiscale solution is: Find $u_{\textnormal{ms}} \in V_{\textnormal{ms}}$ such that

\begin{equation}
a(u_{\textnormal{ms}}, v) = (f, v), \quad \forall v \in V_{\textnormal{ms}}
\end{equation}
In the subsequent sections, we will still use the linear basis functions on $\mathcal{T}^h$ to compute auxiliary space and multiscale space numerically. 
Then, each multiscale basis functions can be linear expressed by $\{\varphi_i\}$, using a discrete vector $\Psi_j$ to denote the coefficient of $j\textnormal{-th}$ multiscale basis. 
We can define an upscaling matrix $R^T$ that stores all the multiscale basis functions (total number is $N_{\textnormal{ms}}$),

$$
R^T = 
\begin{bmatrix}
\Psi_1 & \Psi_2 & \cdots & \Psi_{N_{\textnormal{ms}}}
\end{bmatrix}
$$
The construction of multiscale basis functions will be introduced in Section \ref{sec:cembasis}. Then, the coarse grid solution is $u_H = (R^T A R)^{-1} (R^T F)$, the multiscale solution need to multiply the downscaling matrix $R$, that is $u_{\textnormal{ms}} = R^T u_H.$

\section{The construction of the CEM-GMsFEM basis functions} \label{sec:cembasis}

In this section, we construct multiscale spaces on coarse grid. 
Let $\mathcal{T}^H$ be a conforming partition of the computational domain $\Omega^{\epsilon}$, such that each element is the union of triangle in $\mathcal{T}^h$. Using $H$ represent the coarse mesh size.
Typically, we assume that $0 < h \ll H < \mathrm{diam}(\Omega).$  
Let $N_c$ be the total number of coarse elements and $N_v$ be the total number of vertices of $\mathcal{T}^H$.

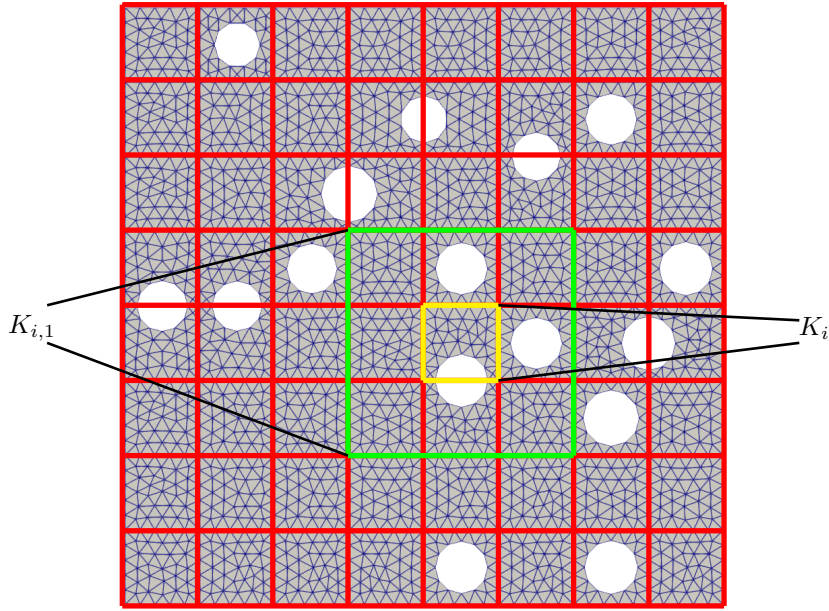
\begin{figure}[htbp]
\centering
\begin{tikzpicture}
\centering
\node[anchor=south west,inner sep=0] at (-4,-4) {\includegraphics[height=8cm]{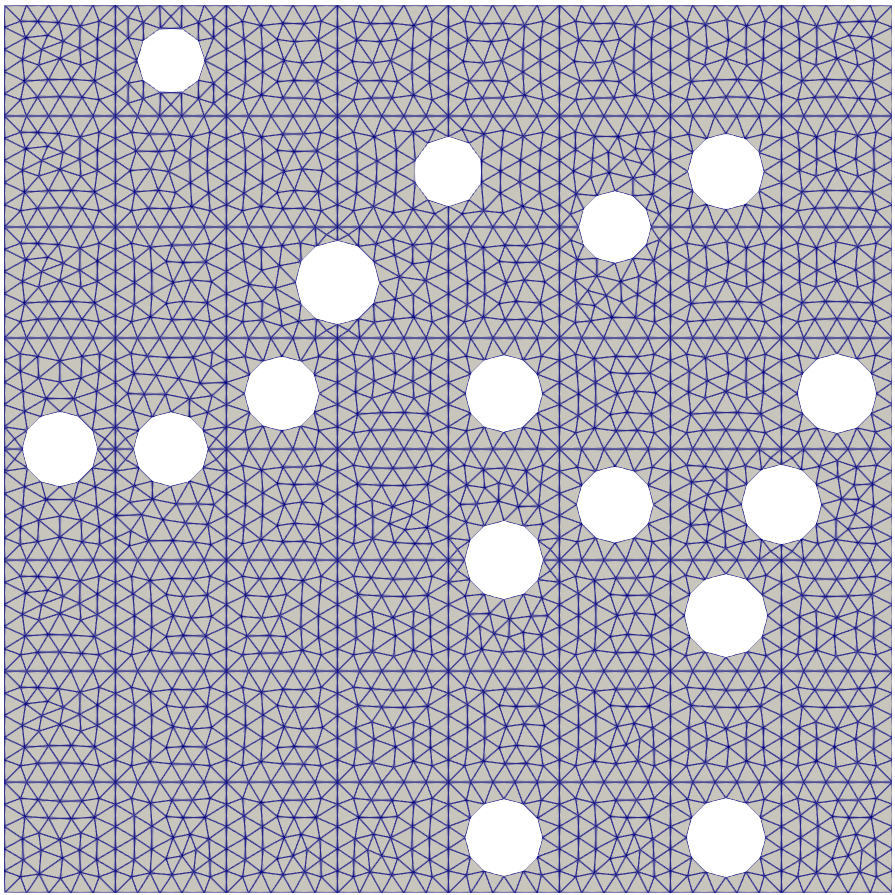}};
\draw[red, line width=2pt] (-4,-4) grid (4,4);
\draw[yellow, line width=2pt] (0,-1) grid (1,0);
\draw[green, line width=2pt] (-1,-2) -- (-1,1);
\draw[green, line width=2pt] (-1,-2) -- (2,-2);
\draw[green, line width=2pt] (-1,1) -- (2,1);
\draw[green, line width=2pt] (2,-2) -- (2,1);
\draw[line width=1pt] (1,0) -- (5,-0.2);
\draw[line width=1pt] (1,-1) -- (5,-0.5);
\node at (5.2,-0.3)  {$K_i$};
\draw[line width=1pt] (-1,1) -- (-5,0);
\draw[line width=1pt] (-1,-2) -- (-5,-0.5);
\node at (-5.2,-0.3)  {$K_{i,1}$};
\end{tikzpicture}  
\caption{Illustration of the coarse grid (Red), the fine-grid (Blue) and the oversampling domain (Green).}
\label{pic:perforateddomain_mesh_example}
\end{figure}

The construction of the multiscale spaces consists of two steps. 
The first step is to constuct auxiliary multiscale spaces which the concept is similar of GMsFEM, we will compute a eigenvalue problem in each coarse block $K_i$.
Then, for each $K_{i,m_i}$ that enlarging $K_i$ by $m_i$ coarse grid layers, we will solving the energy minimizing problems in the oversampled domains. 
In \autoref{pic:perforateddomain_mesh_example}, we illustrates the fine-scale triangulation, coarse element $K_i$ and oversampling subdomains $K_{i,1}$.
If the parameter $m_i$ is appropriately chosen, the error demonstrates linear convergence towards the reference solution in relation to the coarse mesh size $H$. 
This assertion is rigorously established and substantiated in the subsequent section.

\subsection{Auxiliary space}
We will construct the auxiliary multiscale basis functionsin in each coarse block $K_i$. Before started, we need define two inner product $a_i(\cdot, \cdot), ~ s_i(\cdot, \cdot)$ in $K_i$, 

\begin{equation}
a_i(v, w)=\int_{K_i} {\nabla}v{\cdot}{\nabla}w, \quad
s_i(v,w)=\int_{K_i} \tilde{\kappa}vw.
\label{eq:as_innerproduct_define}
\end{equation}
where $\tilde{\kappa}= \sum_{j=1}^{N_v} |\nabla \chi_j|^2$ and $\{\chi_j\}_{j=1}^{N_v}$ are the partition of unity functions \cite{babuvska1997partition} that defined on coarse grid. In particular, the function $\chi_j$ satisfies $|\nabla \chi_j| = O(1/H)$ and $0 \le \chi_j \le 1$, we use the Lagrange basis for simplify computation in here. Consider the following local eigenvalue problem, 

\begin{equation}
a_i(\phi_j^i, v) = 
\lambda_j^i s_i(\phi_j^i, v),
\quad \forall v \in H^1(K_i).
\label{eq:aux_weak}
\end{equation}

Solving \eqref{eq:aux_weak}, arrange the eigenvalues in ascending order such that

$$
0 = \lambda_1^i \le \lambda_2^i \le \cdots \le \lambda_{l_i}^i \le \cdots
$$
for each $i\in \{1, 2, \cdots, N_c\}.$ 
We use the first $l_i$ eigenfunctions to construct the local auxilary space, $V_{\textnormal{aux}}^i = \mathrm{span} \{\phi_i^1, \cdots, \phi_i^{l_i}\}$.
We have to note that $\phi_j^i$ are a set of unit orthogonal functions with $s(\cdot, \cdot)$ inner product.
Based on that, we can define the global auxliary space $V_{\textnormal{aux}}=\oplus_{i} V_{\textnormal{aux}}^i$.
Given a function $\phi_j^i \in V_{\textnormal{aux}}$, we say that a function $\psi$ is $\phi_j^i\textnormal{-orthogonal}$ if 

$$
s(\psi, \phi_j^i)=1, \quad
s(\psi, \phi_{j'}^{i'})=0, ~ 
\textnormal{if } j'\ne j \textnormal{ or } i'\ne i.
$$
where $s(u, v) = \sum_{i=1}^{N_c} s_i(u, v).$
We also define a global operator $\pi : V \to V_{\textnormal{aux}}$,

$$
\pi(u) = 
\sum_{i=1}^{N_c} \sum_{j=1}^{l_i} 
\frac{s_i(u, \phi_j^i)}{s_i(\phi_j^i, \phi_j^i)}\phi_j^i,
\quad \forall u \in V.
$$
In particular, we define $\tilde{V}$ as the null space of the projection
$\pi$, namely, $\tilde{V} = \{v \in V ~|~ \pi(v)=0\}$.

\subsection{Multiscale space}
Similar to the previous work \cite{chung2018constraint, mmrbook}, we can construct two different type of multiscale basis functions on a oversampling domain. 
But differently, we choose different oversampling layers for each coarse block with can gives more flexibility and save computations.
For each coarse block $K_i$, we can extend this region by $m_i$ coarse grid layer and obtain an oversampled region $K_{i,m_i}$ (see \autoref{pic:perforateddomain_mesh_example}). 
Then for each auxiliary function $\phi_j^i \in V_{\textnormal{aux}}^i$, the multiscale basis function $\psi_{j,\textnormal{ms}}^i$ can be defined by

\begin{equation}
\psi_{j,\textnormal{ms}}^i =  \mathrm{argmin}
\{a(\psi,\psi) ~ | ~\psi \in V_0(K_{i,m_i}), ~
\psi ~\textnormal{is}~ \phi_j^i\textnormal{-orthogonal} \}
\label{eq:c_mini}
\end{equation}
where $V_0(K_{i,m_i}) = \{v \in H^1(K_{i,m_i}) ~|~
v=0 \textnormal{ in } \partial K_{i,m_i} / (\partial K_{i,m_i} 
\cap \partial \mathcal{B}^{\epsilon}) \}$.
Refer to \cite{fu2018constraint}, by using Lagrange Multiplier, the problem \eqref{eq:c_mini} can be rewritten as the following problem: find an 
$(\psi_{j,\textnormal{ms}}^i, w) \in V_0(K_{i,m_i}) \times V_{\textnormal{aux}}(K_{i,m_i})$, 

\begin{equation}
\begin{cases}
\begin{aligned}
a(\psi_{j,\textnormal{ms}}^i, v) + s(v,w) &= 0,
\quad \forall v\in V_0(K_{i,m_i}), \\
s(\psi_{j,\textnormal{ms}}^i-\phi_j^i, q) &= 0,
\quad \forall q\in V_{\textnormal{aux}}(K_{i,m_i}).
\label{eq:c_variation}
\end{aligned}
\end{cases}
\end{equation}
where $V_{\textnormal{aux}}(K_{i,m_i})$ is the union of all local auxiliary spaces for $K_j \subset K_{i,m_i}.$ 
In fact, $\psi_{j,\textnormal{ms}}^i$ can be numerically solved on the fine scale mesh. Besides, $w$ belongs to auxilary space, it can be expressed by $\phi_j^i$.
We using $A^i, S^i$ to denote the stiff matrix and mass matrix in $K_{i,m_i}$ that defined by inner product in  \eqref{eq:as_innerproduct_define}. 
The column of matrix $P^i$ is the discrete style of all auxiliary basis function includes in $V_{\textnormal{aux}}(K_{i,m_i})$, the matrix form of \eqref{eq:c_variation} is,

\begin{equation}
\begin{bmatrix}
A^i & S^i P^i \\[0.5em]
(S^i P^i)^T & \boldsymbol{0}
\end{bmatrix}
\begin{bmatrix}
\psi_h^i \\[0.5em] w_h
\end{bmatrix} = 
\begin{bmatrix}
\boldsymbol{0} \\[0.5em] I
\end{bmatrix}
\end{equation}
where $I$ is a sparse matrix whose nonzero elements (all are 1) depends on the index order of the corresponding auxiliary basis function $\phi_j^i$ in $V_{\textnormal{aux}}(K_{i,m_i})$. 
For our convergence analysis, we need to define a global basis functions. 
The global multiscale basis function of constraint version is $\psi_j^i \in V$ is defined in a similar way, namely, 

\begin{equation}
\psi_{j}^i = \mathrm{argmin} \{ a(\psi,\psi) ~|~
\psi \in V, ~ \psi \textnormal{ is }
\phi_j^i \textnormal{-orthogonal}\}.
\label{eq:c_mini_global}
\end{equation}
We can also get the equivalent variational problem,

\begin{equation}
\begin{cases}
\begin{aligned}
a(\psi_j^i, v) + s(v,w) &= 0, \quad \forall v\in V, \\
s(\psi_j^i-\phi_j^i, q) &= 0, \quad \forall q\in V_{\textnormal{aux}}.
\label{eq:c_variation_global}
\end{aligned}
\end{cases}
\end{equation}

Following \cite{chung2018constraint}, we can relax the $\phi\textnormal{-orthogonality}$ in \eqref{eq:c_mini} and get a relaxed version of the multiscale basis functions.
More specifically, we solve the following un-constrainted minimization problem: Find $\psi_{j,\textnormal{ms}}^i \in V_0(K_{i,m_i})$ such that

\begin{equation}
\psi_{j,\textnormal{ms}}^i = \mathrm{argmin}
\{ a(\psi,\psi)+s(\pi(\psi)-\phi_j^i, \pi(\psi)-\phi_j^i)~|~
\psi \in V_0(K_{i,m_i})\}.
\label{eq:r_mini}
\end{equation}
we can also derive the variation form of \eqref{eq:r_mini},

\begin{equation}
a(\psi_{j,\textnormal{ms}}^i, v) + s(\pi(\psi_{j,\textnormal{ms}}^i), \pi(v)) = 
s(\phi_j^i, \pi(v)), \quad 
\forall v \in V_0(K_{i,m_i}).
\label{eq:r_variation}
\end{equation}
Using the same notation before, the matrix form of \eqref{eq:r_variation} is,

\begin{equation}
\left(
A^i + (S^iP^i)(S^{i}P^{i})^T
\right)
\psi_{j,h}^i = P_j^i S^{i,T}
\label{eq:r_matrix}
\end{equation}
where $P_j^i$ is the $j\textnormal{-th}$ eigenfunction in $V_{\textnormal{aux}}^i.$
The relaxed version of global multiscale basis function $\psi_j^i \in V$ is defined in a similar way, namely, 

\begin{equation}
\psi_{j}^i = \mathrm{argmin}
\{ a(\psi,\psi)+s(\pi(\psi)-\phi_j^i, \pi(\psi)-\phi_j^i)~|~
\psi \in V\}.
\label{eq:r_mini_global}
\end{equation}
which is equivalent to the following variational form

\begin{equation}
a(\psi_j^i, v) + s(\pi(\psi_j^i), \pi(v)) = 
s(\phi_j^i, \pi(v)), \quad 
\forall v \in V.
\label{eq:r_variation_global}
\end{equation}

We have to note that multiscale basis functions $\psi_{j,\textnormal{ms}}^i$ is corresponding to auxiliary basis $\phi_j^i$ one by one both for constraint version and relaxed version.
After we construct multiscale basis functions for each $K_{i,m_i}$, we can define the multiscale space, 

$$
V_{\textnormal{ms}} =  \mathrm{span} \{ \psi_{j,\textnormal{ms}}^i ~|~
1 \le j \le l_i, 1 \le i \le N_c \}.
$$
Then the multiscale approximation is

\begin{equation}
a(u_{\textnormal{ms}}, v) = (f, v), 
\quad \forall v \in V_{\textnormal{ms}}.
\label{eq:perforated_galerkin_ms}
\end{equation}

Through this two different type of global multiscale basis function, 
we can also define $V_{\textnormal{glo}}$ is

$$
V_{\textnormal{glo}} = \mathrm{span} \{ \psi_j^i ~|~
1 \le j \le l_i, 1 \le i \le N_c \}.
$$
Then the global approximation is

\begin{equation}
a(u_{\textnormal{glo}}, v) = (f, v), \quad \forall v \in V_{\textnormal{glo}}.
\label{eq:perforated_galerkin_glo}
\end{equation}
This global multiscale finite element space $V_{\textnormal{glo}}$ satisfies a very important orthogonality property, which will be used in our convergence analysis. 
Recall $\tilde{V}$ is the null space of projection $\pi$, for any $v \in \tilde{V}$, constraint version and relaxed version of global multiscale basis function $\psi_j^i$ can conclude, $a(\psi_j^i, v)=0$, which is obviously in \eqref{eq:c_variation_global} and \eqref{eq:r_variation_global}. Therefore, we have $\tilde{V} \subset V_{\textnormal{glo}}^{\bot}.$ 
Since $\mathrm{dim}(V_{\textnormal{glo}}) = \mathrm{dim}(V_{\textnormal{aux}})$, we have $\tilde{V} = V_{\textnormal{glo}}^{\bot}.$
Thus, we can conclude $V = V_{\textnormal{glo}} \oplus \tilde{V}.$

\section{Convergence analysis} \label{sec:analysis}
In this section, we establish the convergence of the method introduced in Section \ref{sec:cembasis}. 
Before proving the convergence of the method, we need to define some notations and some lemma that suitable for both constraint version and relaxed version. 
We will define three different norms. 
The first norm is $L^2\textnormal{-norm} ~ \Vert \cdot \Vert$ where $ \Vert u \Vert = \int_{\Omega^{\epsilon}} u^2.$
Second is the $a\textnormal{-norm} ~ \Vert\cdot\Vert_a$
where $\Vert u \Vert_a^2 = 
\int_{\Omega^{\epsilon}} |\nabla u|^2$, 
while the last is $s\textnormal{-norm} ~ \Vert\cdot\Vert_s$
where $\Vert u \Vert_s^2 = 
\int_{\Omega^{\epsilon}} \tilde{\kappa} u^2.$
For a given subdomain $\omega \subset \Omega^{\epsilon}$,
we will define the local $a\textnormal{-norm}$ and $s\textnormal{-norm}$ 
by $\Vert u \Vert_{a(\omega)}^2 = 
\int_{\omega} |\nabla u|^2$ and
$\Vert u \Vert_{s(\omega)}^2 = 
\int_{\omega} \tilde{\kappa} u^2.$

In this work, we need to define $\Lambda, \Gamma$ as the following.

$$
\Lambda_{\omega} = \min_{K_i \cap \omega \ne \emptyset} \lambda_{l_i+1}^i, \quad
\Gamma_{\omega} = \max_{K_i \cap \omega \ne \emptyset} \lambda_{l_i}^i,
$$
where $\omega$ is a subset of $\Omega^{\epsilon}$.
Similiar with \cite{chung2018constraint, chung2021convergence}, we need the cutoff function to estimate the difference between the global and multiscale basis function. 
For each $K_i$, consider that $K_{i,m_i} \subset \Omega^{\epsilon}$ denotes the oversampling coarse region obtained by enlarging $K_i$ with $m_i$ additional coarse layers.
For $M_i>m_i$, we define $\chi_i^{M_i,m_i} \in \mathrm{span} 
\{\chi_j\}_{j=1}^{N_v}$ such that $0\le \chi_i^{M_i,m_i} \le 1$ and 

\begin{equation}
\chi_i^{M_i,m_i} = 
\begin{cases}
\begin{aligned}
1~ &\textnormal{ in } K_{i,m_i}, \\
0~ &\textnormal{ in } \Omega^{\epsilon} \setminus K_{i,M_i}. \\
\end{aligned}
\end{cases}
\label{eq:def_cutoff}
\end{equation}
Note that, we have $K_{i,m_i} \subset K_{i,M_i} \subset \Omega^{\epsilon}.$
The prove is very similar to the previous work \cite{chung2018constraint, chung2021convergence}, but we will prove the oversampling layers $k_i$ is also dependent of the eigenvalues in the oversampling domain, and define $k = \max_{1\le i \le N_c} k_i$.
Before analyse the multiscale solution, we need some estimate between different norms for later use in the analysis.

\begin{lemma} \label{lemma:aux_property}
Let $k_i \ge 2$ be an integer and define $W \coloneqq \{ v \in V ~|~ v|_{K_i} \notin V_{\textnormal{aux}}^i \}$. Then, the following inequalities hold

\begin{enumerate}[label=(\roman*)]
\item if $v \in V_{\textnormal{aux}}^i$, $\Vert v \Vert_{a(K_i)}^2 \le \lambda_{l_i}^i \Vert v \Vert_{s(K_i)}^2.$
\item if $v \in W$, $\Vert v \Vert_s^2 \le \Lambda_{\mathrm{supp}(v)}^{-1} \Vert v \Vert_a^2.$
\item if $v \in V$, $\Vert v \Vert_s^2 \le \Lambda_{\mathrm{supp}(v)}^{-1} \Vert (I-\pi) v \Vert_a^2 
+ \Vert \pi(v) \Vert_s^2.$
\item if $v \in V$, 
$$
\Vert (1-\chi_i^{k_i,k_i-1}) v \Vert_a^2 \le 2(1+\Lambda_{\mathrm{supp}(v)}^{-1}) 
\Vert v \Vert_{a(\Omega^{\epsilon}\setminus K_{i,k_i-1})}^2 + 
2\Vert \pi(v) \Vert_{s(\Omega^{\epsilon} \setminus K_{i,k_i-1})}^2.
$$
\item if $v \in V$, $\Vert (1-\chi_i^{k_i,k_i-1}) v \Vert_s^2 \le \Lambda_{\mathrm{supp}(v)}^{-1}
\Vert v \Vert_{a(\Omega^{\epsilon}\setminus K_{i,k_i-1})}^2 + 
\Vert \pi(v) \Vert_{s(\Omega^{\epsilon} \setminus K_{i,k_i-1})}^2.$
\end{enumerate}
\end{lemma}
\begin{proof}
For any $z \in H^1(K_i)$, we can present $z = \sum_{j\ge 1} \alpha_j^i \phi_j^i$ with $\alpha_j^i \in \mathbb{R}.$ 

\begin{enumerate}[label=(\roman*)]
\item 
Since $v \in V_{\textnormal{aux}}^i$, then $\alpha_j^i=0$
for $j\ge l_i+1.$ Recall the eigenfunctions that using local spectral problem 
\eqref{eq:aux_weak} are orthogonal to each other, we obtain

$$
\Vert v \Vert_{a(K_i)}^2 =
\sum_{j=1}^{l_i} \alpha_j^i \lambda_j^i
s(\phi_j^i, v) \le \lambda_{l_i}^i \sum_{j=1}^{l_i} 
\alpha_j^i s(\phi_j^i, v) = \lambda_{l_i}^i \Vert v \Vert_{s(K_i)}^2.
$$
\item For any $v \in W$, 
then $v = \sum_{i=1}^{N_c} \sum_{j\ge l_i+1} \alpha_j^i \phi_j^i.$

$$
\begin{aligned}
\Vert v \Vert_a^2 
=& \sum_{K_i \subset \mathrm{supp}(v)} \sum_{j \ge l_i+1} \alpha_j^i a_i(\phi_j^i, v) \\
=& \sum_{K_i \subset \mathrm{supp}(v)} \sum_{j \ge l_i+1} \alpha_j^i \lambda_j^i s_i(\phi_j^i, v) \\
\ge&  \Lambda_{\mathrm{supp}(v)} \sum_{i=1}^{N_c} \sum_{j \ge l_i+1}
\alpha_j^i s(\phi_j^i, v)  \\
=& \Lambda_{\mathrm{supp}(v)} \Vert v \Vert_s^2.
\end{aligned}
$$
\item For any $v \in V$,
$$
\begin{aligned}
\Vert v \Vert_s^2 &\le \Vert (I-\pi)v \Vert_s^2 +
\Vert \pi(v) \Vert_s^2 \\
&\le \Lambda_{\mathrm{supp}(v)}^{-1} \Vert (I-\pi)v \Vert_a^2 +
\Vert \pi(v) \Vert_s^2.
\end{aligned}
$$
where operator $I$ means the identity operator, 
the last inequality is follows from (ii).
\item By using the property of cutoff function 
$\chi_i^{k_i,k_i-1}$ and (iii),

$$
\begin{aligned}
&\Vert (1-\chi_i^{k_i,k_i-1}) v \Vert_a^2 \\
=& \int_{\Omega^{\epsilon} \setminus K_{i,k_i-1}}
\left|\nabla 
\left( (1-\chi_i^{k_i,k_i-1}) v \right)\right|^2 \\
\le& 2
\int_{\Omega^{\epsilon} \setminus K_{i,k_i-1}} 
(1-\chi_i^{k_i,k_i-1})^2 |\nabla v |^2 + 
|v \nabla \chi_i^{k_i,k_i-1}|^2 \\
\le& 2(
\Vert v \Vert_{a(\Omega^{\epsilon} \setminus K_{i,k_i-1})}^2 +
\Vert v \Vert_{s(\Omega^{\epsilon} \setminus K_{i,k_i-1})}^2) \\
\le& 2(1+\Lambda_{\Omega^{\epsilon} \setminus K_{i,k_i-1}}^{-1}) 
\Vert v \Vert_{a(\Omega^{\epsilon}\setminus K_{i,k_i-1})}^2 + 
2\Vert \pi(v) \Vert_{s(\Omega^{\epsilon} \setminus K_{i,k_i-1})}^2
\end{aligned}
$$
\item For any $k_i \ge 2$ and combine (iii), we have

$$
\begin{aligned}
\Vert (1-\chi_i^{k_i,k_i-1}) v \Vert_s^2 &\le 
\Vert v \Vert_{s(\Omega^{\epsilon} \setminus K_{i,k_i-1})}^2 \\
&\le \Lambda_{\Omega^{\epsilon} \setminus K_{i,k_i-1}}^{-1} \Vert v 
\Vert_{a(\Omega^{\epsilon} \setminus K_{i,k_i-1})}^2 + 
\Vert \pi(v) \Vert_{s(\Omega^{\epsilon} \setminus K_{i,k_i-1})}^2.
\end{aligned}
$$
\end{enumerate}
This completes the proof.
\end{proof}

To prove the convergence result of the proposed method, we need to recall the  convergence result of using the global multiscale basis functions in \cite{chung2018constraint}. 

\begin{lemma} \label{lemma:uuglo_esti}
Let $u$ be the solution of \eqref{eq:perforated_variation} and 
$u_{\textnormal{glo}}$ be the solution of \eqref{eq:perforated_galerkin_glo}.
We have 

$$
\Vert u - u_{\textnormal{glo}} \Vert_a \le C \Lambda_{\Omega^{\epsilon}}^{-\frac{1}{2}}
\Vert \tilde{\kappa}^{-\frac{1}{2}}f \Vert.
$$
Moreover, if $\{\chi_j\}_{j=1}^{N_c}$ is a set of bilinear 
partition of unity, we have

$$
\Vert u - u_{\textnormal{glo}} \Vert_a \le C
H \Lambda_{\Omega^{\epsilon}}^{-\frac{1}{2}} \Vert f \Vert.
$$
\end{lemma}

To achieve convergence, we must assess the disparity between global and local multiscale basis functions. Given the distinct construction methods of constraint version and relaxed version, separate estimations are necessary. The estimation of the distinction between global and multiscale basis functions relies on a lemma applicable to both versions of multiscale basis functions. For each coarse block $K$, we define $B$ as a bubble function with $B(\boldsymbol{x}) > 0$ for all $\boldsymbol{x} \in K$ and $B(\boldsymbol{x})=0$ for all $\boldsymbol{x} \in \partial K.$
We will take $B=\Pi_j \chi_j$ where the product is taken over all
vertices $j$ on the boundary of $K.$ Using the bubble function,
we define the constant 

$$
C_{\pi} = \sup_{K \in \mathcal{T}^H, \mu \in V_{\textnormal{aux}}}
\frac{\int_K \tilde{\kappa}\mu^2}{\int_K B\tilde{\kappa}\mu^2}.
$$

\begin{lemma} \label{lemma:vaux_v}
For all $v_{\textnormal{aux}} \in V_{\textnormal{aux}}$,
there exists a function $v \in V$ such that

$$
\pi(v)=v_{\textnormal{aux}}, \quad
\Vert v \Vert_a^2 \le D_{\mathrm{supp}(v)} \Vert v_{\textnormal{aux}} \Vert_s^2, \quad
\mathrm{supp} (v) \subset \mathrm{supp}(v_{\textnormal{aux}}),
$$
where $D_{\mathrm{supp}(v)}=C_{\mathcal{T}}(1+\Gamma_{\mathrm{supp}(v)})$,
and $C_{\mathcal{T}}$ is the square of the maximum number of 
vertices over all coarse elements.
\end{lemma}
\begin{proof} 
Without loss of generality, we can assume that $v_{\textnormal{aux}} \in V_{\textnormal{aux}}^i.$
Consider the following variational problem: find $v \in V(K_i)$ and 
$\mu \in V_{\textnormal{aux}}^i$ such that

\begin{equation}
\begin{cases}
\begin{aligned}
a_i(v, w) + s_i(w, \mu) &= 0, \quad \forall w\in V(K_{i}), \\
s_i(v-v_{\textnormal{aux}}, q) &= 0, \quad 
\forall q\in V_{\textnormal{aux}}^i.
\end{aligned}
\end{cases}
\label{eq:variation_Ki_p}
\end{equation}
Note that, the well-posedness of the problem \eqref{eq:variation_Ki_p} is equivalent to the existence of a function $v \in V(K_i)$ such that 

$$
s_i(v, v_{\textnormal{aux}}) \ge 
C_1 \Vert v_{\textnormal{aux}} \Vert_{s(K_i)}^2, \quad
\Vert v \Vert_{a(K_i)} \le C_2 \Vert v_{\textnormal{aux}} \Vert_{s(K_i)}
$$
where $C_1$ and $C_2$ are the constants to be determined.

Note that $v_{\textnormal{aux}}$ is supported in $K_i.$
We let $v=B v_{\textnormal{aux}}.$ By the definition of $s_i$,
we have 

$$
s_i(v, v_{\textnormal{aux}}) = \int_{K_i}\tilde{\kappa} 
B v_{\textnormal{aux}}^2 \ge C_{\pi}^{-1} 
\Vert v_{\textnormal{aux}} \Vert_{s(K_i)}^2.
$$
Since $\nabla(B v_{\textnormal{aux}}) = v_{\textnormal{aux}} \nabla B + 
B \nabla v_{\textnormal{aux}}, |B| \le 1$ and 
$|\nabla B|^2 \le C_{\mathcal{T}} \sum_j |\nabla \chi_j|^2$,
we have

$$
\Vert v \Vert_{a(K_i)}^2 = 
a_i(v, B v_{\textnormal{aux}}) \le 
C_{\mathcal{T}} \Vert v \Vert_{a(K_i)}
\left(
\Vert v_{\textnormal{aux}} \Vert_{a(K_i)} + 
\Vert v_{\textnormal{aux}} \Vert_{s(K_i)}
\right).
$$
Combining \cref{lemma:aux_property} (i), the existence and uniqueness of the function $v$ can be deduced for a given auxiliary function $v_{\textnormal{aux}} \in V_{\textnormal{aux}}^i.$ It is evident from the second equality in \eqref{eq:variation_Ki_p} that $\pi(v) = v_{\textnormal{aux}}.$ The remaining two conditions in the lemma naturally follow from the preceding proof. It is important to note that the constant $D$ depends on the eigenvalues that in support of $v_{\textnormal{aux}}.$
\end{proof}

Next, we will separately establish the remaining components of the convergence analysis for both the constraint version and the relaxed version.
That is, our goal is to estimate the difference of multiscale basis functions and global multiscale basis functions.

\subsection{Constraint Version}
Before we start analyze the constraint version, we need to estimate the difference of multiscale basis functions and global multiscale basis functions are dependent of the oversampling layers. The following lemma has prove the exponential decay property. 

\begin{lemma} \label{lemma:glolocal_c}
We consider the oversampled domain $K_{i,k_i}$ with $k_i\ge 2.$
That is, $K_{i,k_i}$ is an oversampled region by enlarging
$K_i$ by $k_i$ coarse grid layers.
Let $\phi_j^i \in V_{\textnormal{aux}}$ be a given auxiliary 
multiscale basis function. We let $\psi_{j,\textnormal{ms}}^i$ 
be the multiscale basis functions obtained in \eqref{eq:c_mini}
and let $\psi_j^i$ be the global multiscale basis functions 
obtained in \eqref{eq:c_mini_global}. Then we have

$$
\Vert \psi_j^i - \psi_{j,\textnormal{ms}}^i \Vert_a^2 \le 
E_i \Vert \phi_j^i \Vert_{s(K_i)}^2
$$
where $E_i=8D_{K_{i,k_i}}^2(1+\Lambda_{\Omega^{\epsilon} \setminus K_{i,k_i}}^{-1})
\left(
1 + \frac{\Lambda_{K_{i,k_i-1}}^{\frac{1}{2}}}{2D_{K_{i,k_i-1}}^{\frac{1}{2}}}
\right)^{1-k_i}.$
\end{lemma}
\begin{proof}
For the given $\phi_j^i \in V_{\textnormal{aux}}$, 
using \cref{lemma:vaux_v}, there exists a $\tilde{\phi}_j^i \in V$ such that

\begin{equation}
\pi(\tilde{\phi}_j^i) = \phi_j^i, \quad
\Vert \tilde{\phi}_j^i \Vert_a^2 \le D_{K_i} 
\Vert \phi_j^i \Vert_{s(K_i)}^2, 
\textnormal{ and } \quad
\mathrm{supp}(\tilde{\phi}_j^i) \subset K_i.
\label{eq:def_tildephi}
\end{equation}
We let $\eta = \psi_j^i - \tilde{\phi}_j^i.$
Note that $\eta \in \tilde{V}$ since $\pi(\eta)=0.$
By using the resulting variational forms of minimization problems
\eqref{eq:c_mini} and \eqref{eq:c_mini_global}, we see that
$\psi_{j,\textnormal{ms}}^i$ and $\psi_j^i$ satisfy 

\begin{equation}
a(\psi_{j,\textnormal{ms}}^i, v) + 
s(v, \mu_{j,\textnormal{ms}}^i) = 0, 
\quad \forall v \in V_0(K_{i,k_i})
\label{eq:c_weak_lem_decay}
\end{equation}
where $V_0(K_{i,k_i}) = \{v \in H^1(K_{i,k_i}) ~|~
v=0 \textnormal{ in } \partial K_{i,k_i} / (\partial K_{i,k_i} 
\cap \partial \mathcal{B}^{\epsilon}) \}$, and 

\begin{equation}
a(\psi_j^i, v) + s(v, \mu_j^i) = 0, \quad
\forall v \in V
\label{eq:c_weak_lem_glo_decay}
\end{equation}
for some $\mu_{j,\textnormal{ms}}^i, \mu_j^i \in V_{\textnormal{aux}}.$
Define a null space in $K_{i,k_i}$, 

$$
\tilde{V}_0(K_{i,k_i}) = \{v \in V_0(K_{i,k_i}) ~|~ \pi(v)=0 \}.
$$
Subtracting the above two equations \eqref{eq:c_weak_lem_decay} and \eqref{eq:c_weak_lem_glo_decay}, and restricting 
$v \in \tilde{V}_0(K_{i,k_i})$, we have

$$
a(\psi_j^i-\psi_{j,\textnormal{ms}}^i, v) = 0, \quad
\forall v \in \tilde{V}_0(K_{i,k_i}).
$$
Therefore, for $v\in \tilde{V}_0(K_{i,k_i})$ and combine $(-\psi_{j,\textnormal{ms}}^i+\tilde{\phi}_j^i) \in \tilde{V}(K_{i,k_i})$, we have

$$
\begin{aligned}
\Vert \psi_j^i - \psi_{j,\textnormal{ms}}^i \Vert_a^2
&= a(\psi_j^i - \psi_{j,\textnormal{ms}}^i, 
\psi_j^i - \psi_{j,\textnormal{ms}}^i) \\
&= a(\psi_j^i - \psi_{j,\textnormal{ms}}^i, 
\psi_j^i - \tilde{\phi}_j^i 
- \psi_{j,\textnormal{ms}}^i + \tilde{\phi}_j^i) \\
&= a(\psi_j^i - \psi_{j,\textnormal{ms}}^i, \eta-v).
\end{aligned}
$$
Hence, we conclude

\begin{equation}
\Vert \psi_j^i - \psi_{j,\textnormal{ms}}^i \Vert_a \le 
\Vert \eta - v \Vert_a
\label{ineq:msglo_etav}
\end{equation}
for all $v \in \tilde{V}_0(K_{i,k_i}).$
For $i\textnormal{-th}$ coarse block $K_i$, we consider two oversampled regions $K_{i,k_i-1}$ and $K_{i,k_i}.$ 
We define the cutoff function $\chi_i^{k_i,k_i-1}$ with the properties in \eqref{eq:def_cutoff}, where $M_i=k_i$ and $m_i=k_i-1.$
It follows that $\chi_i^{k_i,k_i-1} \equiv 1$ for any $K_j \subset K_{i,k_i-1}.$ 
Since $\eta \in \tilde{V}$, we have

$$
s_j(\chi_i^{k_i,k_i-1} \eta, \phi_n^j) = 
s_j(\eta, \phi_n^j) = 0, \quad
\forall n=1,2,\cdots, l_j.
$$
From the above result and the fact that $\chi_i^{k_i,k_i-1} 
\equiv 0 \textnormal{ in } \Omega^{\epsilon} \setminus K_{i,k_i}$,
we have

$$
\mathrm{supp} 
\left( \pi(\chi_i^{k_i,k_i-1} \eta) \right) \subset
K_{i,k_i} \setminus K_{i,k_i-1}.
$$
Using \cref{lemma:vaux_v}, for the function 
$\pi(\chi_i^{k_i,k_i-1} \eta)$, there is $\mu \in V$ such that 
$\mathrm{supp}(\mu) \subset K_{i,k_i} \setminus K_{i,k_i-1}$ and
$\pi(\mu-\chi_i^{k_i,k_i-1}\eta)=0.$ Moreover, also from 
\cref{lemma:vaux_v},

\begin{equation}
\begin{aligned}
\Vert \mu \Vert_{a(K_{i,k_i}\setminus K_{i,k_i-1})} 
&\le D_{K_{i,k_i}\setminus K_{i,k_i-1}}^{\frac{1}{2}} \Vert \pi(\chi_i^{k_i,k_i-1}\eta) 
\Vert_{s(K_{i,k_i}\setminus K_{i,k_i-1})} \\
&\le D_{K_{i,k_i}\setminus K_{i,k_i-1}}^{\frac{1}{2}} \Vert \chi_i^{k_i,k_i-1}\eta 
\Vert_{s(K_{i,k_i}\setminus K_{i,k_i-1})}
\end{aligned}
\label{ineq:mu_chieta}
\end{equation}
where the last inequality follows from the fact that $\pi$ is a projection. 
Taking $v=\mu + \chi_i^{k_i,k_i-1}\eta$ in \eqref{ineq:msglo_etav}, we have

\begin{equation}
\Vert \psi_j^i - \psi_{j,\textnormal{ms}}^i \Vert_a \le
\Vert \eta - v \Vert_a \le \Vert (1-\chi_i^{k_i,k_i-1})\eta \Vert_a 
+ \Vert \mu \Vert_{a(K_{i,k_i}\setminus K_{i,k_i-1})}.
\label{ineq:msglo_etav_2}
\end{equation}
Next, we will use three steps to estimate the two terms on the right hand side. 

\paragraph{Step 1:}
We will prove the two terms on the right hand side in \eqref{ineq:msglo_etav_2} can be bounded by $\Vert \eta \Vert_{a(\Omega^{\epsilon} \setminus K_{i,k_i-1})}$. 
For the first term in \eqref{ineq:msglo_etav_2}, $\Vert (1-\chi_i^{k_i,k_i-1})\eta \Vert_a$. 
From \cref{lemma:aux_property} (iv) and $\eta \in \tilde{V}$, we have

$$
\Vert (1-\chi_i^{k_i,k_i-1})\eta \Vert_a^2 \le  2(1+\Lambda_{\Omega^{\epsilon} \setminus K_{i,k_i-1}}^{-1}) 
\Vert \eta \Vert_{a(\Omega^{\epsilon} \setminus K_{i,k_i-1})}^2.
$$
For the second term in \eqref{ineq:msglo_etav_2}, $\Vert \mu \Vert_{a(K_{i,k_i}\setminus K_{i,k_i-1})}.$ By \eqref{ineq:mu_chieta} and \cref{lemma:aux_property} (ii), we have

$$
\begin{aligned}
\Vert \mu \Vert_{a(K_{i,k_i}\setminus K_{i,k_i-1})}^2 &\le 
D_{K_{i,k_i}\setminus K_{i,k_i-1}} \Vert \chi_i^{k_i,k_i-1} \eta \Vert_{s(K_{i,k_i}\setminus K_{i,k_i-1})}^2 \\
&\le \frac{D_{K_{i,k_i}\setminus K_{i,k_i-1}}}{\Lambda_{K_{i,k_i}\setminus K_{i,k_i-1}}} \Vert \eta \Vert_{a(K_{i,k_i}\setminus K_{i,k_i-1})}^2.
\end{aligned}
$$
Thus, we obtain
\begin{equation}
\Vert \psi_j^i - \psi_{j,\textnormal{ms}}^i \Vert_a^2 \le 
2D_{K_{i,k_i}\setminus K_{i,k_i-1}}(1+\Lambda_{\Omega^{\epsilon}\setminus K_{i,k_i-1}}^{-1}) \Vert \eta 
\Vert_{a(\Omega^{\epsilon} \setminus K_{i,k_i-1})}^2.
\label{ineq:msglo_etak1}
\end{equation}

\paragraph{Step 2:}
We will prove the following recursive inequality,

\begin{equation}
\Vert \eta \Vert_{a(\Omega^{\epsilon} \setminus K_{i,k_i-1})}^2
\le 
\left(
1 + \frac{\Lambda_{K_{i,k_i-1}\setminus K_{i,k_i-2}}^{\frac{1}{2}}}{2D_{K_{i,k_i-1}\setminus K_{i,k_i-2}}^{\frac{1}{2}}}
\right)^{-1} \Vert \eta 
\Vert_{a(\Omega^{\epsilon} \setminus K_{i,k_i-2})}^2.
\label{ineq:recu_k1k2}
\end{equation}
where $k_i-2\ge 0.$ 
Let $\xi = 1-\chi_i^{k_i-1,k_i-2}.$
Then we see that $\xi \equiv 1$ in 
$\Omega^{\epsilon} \setminus K_{i,k_i-1}$ and $0 \le \xi \le 1$
otherwise. Then we have

\begin{equation}
\Vert \eta \Vert_{a(\Omega^{\epsilon} \setminus K_{i,k_i-1})}^2 \le
\int_{\Omega^{\epsilon}} \xi^2 |\nabla \eta|^2 = 
\int_{\Omega^{\epsilon}} \nabla \eta \cdot \nabla (\xi^2 \eta) 
- 2 \int_{\Omega^{\epsilon}} \xi \eta \nabla \xi \cdot \eta.
\label{eq:eta_xieta}
\end{equation}
We estimate the first term in \eqref{eq:eta_xieta}. 
For the function $\pi(\xi^2\eta)$, using \cref{lemma:vaux_v}, 
there exist $\gamma \in V$ such that 
$\pi(\gamma)=\pi(\xi^2\eta)$ and 
$\mathrm{supp}(\gamma) \subset \mathrm{supp}(\pi(\xi^2\eta)).$
For any coarse element $K_m \subset \Omega^{\epsilon} 
\setminus K_{i,k_i-1}$, since $\xi \equiv 1$ on $K_m$, we have 

$$
s_m(\xi^2\eta, \phi_n^m) = 0, \quad \forall n = 1, 2, \cdots, l_m.
$$
On the other hand, since $\xi \equiv 0$ in $K_{i,k_i-2}.$
For any coarse element $K_m \subset K_{i,k_i-2}$, we have

$$
s_m(\xi^2\eta, \phi_n^m) = 0, \quad \forall n = 1, 2, \cdots, l_m.
$$
From the above two conditions, we see that 
$\mathrm{supp}(\pi(\xi^2 \eta)) \subset K_{i,k_i-1}\setminus 
K_{i,k_i-2}$, and consequently 
$\mathrm{supp}(\pi(\gamma)) \subset K_{i,k_i-1}\setminus K_{i,k_i-2}.$
Note that, since $\pi(\gamma)=\pi(\xi^2\eta)$, we have
$\xi^2\eta - \gamma \in \tilde{V}.$ We note also that 
$\mathrm{supp}(\xi^2\eta - \gamma) \subset \Omega^{\epsilon} 
\setminus K_{i,k_i-2}.$ By \eqref{eq:def_tildephi}, the functions
$\tilde{\phi}_j^i$ and $\xi^2\eta - \gamma$ have disjoint 
supports, so $a(\tilde{\phi}_j^i, \xi^2\eta-\gamma)=0.$
Then, by the definition of $\eta$, we have

$$
a(\eta, \xi^2\eta-\gamma) = 
a(\psi_j^i, \xi^2\eta-\gamma).
$$
Since $\xi^2\eta - \gamma \in \tilde{V} = V_{\textnormal{glo}}^{\bot}$, we have
$a(\psi_j^i, \xi^2\eta-\gamma)=0.$ 
Then we can estimate the first term in \eqref{eq:eta_xieta}
as follows

$$
\begin{aligned}
\int_{\Omega^{\epsilon}} \nabla \eta \cdot \nabla (\xi^2 \eta)
&= \int_{\Omega^{\epsilon}} \nabla \eta \cdot \nabla \gamma \\
&\le D_{K_{i,k_i-1}\setminus K_{i,k_i-2}}^{\frac{1}{2}} 
\Vert \eta \Vert_{a(K_{i,k_i-1}\setminus K_{i,k_i-2})}
\Vert \pi(\xi^2\eta) \Vert_{s(K_{i,k_i-1}\setminus K_{i,k_i-2})}
\end{aligned}
$$
where the last inequality follows from \cref{lemma:vaux_v}.
For all coarse elements $K_j \subset K_{i,k_i-1} \setminus K_{i,k_i-2}$,
since $\pi(\eta)=0$, combine \cref{lemma:aux_property} (ii), we have

$$
\Vert \pi(\xi^2\eta) \Vert_{s(K_j)}^2 \le 
\Vert \xi^2 \eta \Vert_{s(K_j)}^2 \le 
\Lambda_{K_j}^{-1} \Vert \eta \Vert_{a(K_j)}^2.
$$
Summing the above over all coarse elements 
$K_j \subset K_{i,k_i-1} \setminus K_{i,k_i-2}$, we have

$$
\Vert \pi(\xi^2\eta) \Vert_{s(K_{i,k_i-1}\setminus K_{i,k_i-2})} \le
 \Lambda_{K_{i,k_i-1}\setminus K_{i,k_i-2}}^{-\frac{1}{2}}
\Vert \eta \Vert_{a(K_{i,k_i-1}\setminus K_{i,k_i-2})}.
$$
To estimate the second term in \eqref{eq:eta_xieta}, 
using triangle inequality and \cref{lemma:aux_property} (ii),

$$
\begin{aligned}
2 \int_{\Omega^{\epsilon}} \xi \eta \nabla\xi \cdot \nabla\eta 
&\le 2 \Vert \eta \Vert_{s(K_{i,k_i-1} \setminus K_{i,k_i-2})}
\Vert \eta \Vert_{a(K_{i,k_i-1} \setminus K_{i,k_i-2})} \\
&\le \Lambda_{K_{i,k_i-1}\setminus K_{i,k_i-2}}^{-\frac{1}{2}}
\Vert \eta \Vert_{a(K_{i,k_i-1} \setminus K_{i,k_i-2})}^2.
\end{aligned}
$$
Hence, by using the above results, \eqref{eq:eta_xieta} can be 
estimated as 

$$
\Vert \eta \Vert_{a(\Omega^{\epsilon} \setminus K_{i,k_i-1})}^2 \le
\frac{2D_{K_{i,k_i-1}\setminus K_{i,k_i-2}}^{\frac{1}{2}}}{\Lambda_{K_{i,k_i-1}\setminus K_{i,k_i-2}}^{\frac{1}{2}}} 
\Vert \eta \Vert_{a(K_{i,k_i-1} \setminus K_{i,k_i-2})}^2.
$$
By using the above inequality, we have

$$
\begin{aligned}
\Vert \eta \Vert_{a(\Omega^{\epsilon} \setminus K_{i,k_i-2})}^2 &=
\Vert \eta \Vert_{a(\Omega^{\epsilon} \setminus K_{i,k_i-1})}^2 + 
\Vert \eta \Vert_{a(K_{i,k_i-1}) \setminus K_{i,k_i-2})}^2 \\
&\ge \left(
1 + \frac{\Lambda_{K_{i,k_i-1}\setminus K_{i,k_i-2}}^{\frac{1}{2}}}{2D_{K_{i,k_i-1}\setminus K_{i,k_i-2}}^{\frac{1}{2}}} \right)
\Vert \eta \Vert_{a(\Omega^{\epsilon} \setminus K_{i,k_i-1})}^2
\end{aligned}
$$

\paragraph{Step 3:}
Finally, we will estimate the term 
$\Vert \eta \Vert_{a(\Omega^{\epsilon} \setminus K_{i,k_i-1})}.$
Using \eqref{ineq:msglo_etak1} and 
\eqref{ineq:recu_k1k2}, we conclude that

\begin{equation}
\begin{aligned}
&\Vert \psi_j^i - \psi_{j,\textnormal{ms}}^i \Vert_a^2 \\
\le&
2D_{K_{i,k_i}\setminus K_{i,k_i-1}}(1+\Lambda_{\Omega^{\epsilon}\setminus K_{i,k_i-1}}^{-1}) 
\left(
1 + \frac{\Lambda_{K_{i,k_i-1}}^{\frac{1}{2}}}{2D_{K_{i,k_i-1}}^{\frac{1}{2}}}
\right)^{-1} 
\Vert \eta \Vert_{a(\Omega^{\epsilon} \setminus K_{i,k_i-2})}^2 \\
\le&
2D_{K_{i,k_i}\setminus K_{i,k_i-1}}(1+\Lambda_{\Omega^{\epsilon}\setminus K_{i,k_i-1}}^{-1}) 
\left(
1 + \frac{\Lambda_{K_{i,k_i-1}}^{\frac{1}{2}}}{2D_{K_{i,k_i-1}}^{\frac{1}{2}}}
\right)^{1-k_i} 
\Vert \eta \Vert_{a}^2.    
\end{aligned}
\end{equation}
By \eqref{eq:def_tildephi} and the definition of $\psi_j^i$, we have 

$$
\Vert \eta \Vert_a = 
\Vert \psi_j^i - \tilde{\phi}_j^i \Vert_a \le 
2 \Vert \tilde{\phi}_j^i \Vert_a \le 
2 D_{K_i}^{\frac{1}{2}} \Vert \phi_j^i \Vert_{s(K_i)}.
$$
This completes the proof.
\end{proof}

The above lemma shows the global basis is localizable. 
We will need one more result before we prove the final
convergence estimate, refer to \cite{chung2018constraint}.

\begin{lemma} \label{lemma:psi_gloms_reccu}
With the same notations in \cref{lemma:glolocal_c}, 
we have 

\begin{equation}
\Vert \sum_{j=1}^{l_i} (\psi_j^i - \psi_{j,\text{ms}}^i) 
\Vert_a^2 \le C  (k+1)^d 
\Vert \sum_{j=1}^{l_i} (\psi_j^i - \psi_{j,\text{ms}}^i) \Vert_a^2.
\label{eq:psi_gloms_recu}
\end{equation}
\end{lemma}

Next, we will use the above lemma to estimate of the error betweem the solution $u$ and the constraint version multiscale solution $u_{\textnormal{ms}}$.

\begin{theorem} \label{thm:convergence_c}
Let $u$ be the solution of \eqref{eq:perforated_variation}  and 
$u_{\textnormal{ms}}$ be the solution of \eqref{eq:perforated_galerkin_ms}.
Define $E = \max_{1\le i \le N_c} E_i$. We have 

$$
\Vert u - u_{\textnormal{ms}} \Vert \le C \Lambda_{\Omega^{\epsilon}}^{-\frac{1}{2}}
\Vert \tilde{\kappa}^{-\frac{1}{2}} f \Vert + 
C (k+1)^{\frac{d}{2}} E^{\frac{1}{2}} \Vert u_{\textnormal{glo}} \Vert_s,
$$
where $u_{\textnormal{glo}}$ is the solution of 
\eqref{eq:perforated_galerkin_glo}.
Moreover, if each oversampling parameter $k_i$ is sufficiently large 
and $\{\chi_i\}_{i=1}^{N_v}$ is a set of bilinear 
partition of unity, we have 

$$
\Vert u - u_{\textnormal{ms}} \Vert_a \le 
C H \Lambda_{\Omega^{\epsilon}}^{-\frac{1}{2}} \Vert f \Vert.
$$
\end{theorem}
\begin{proof}
We write $u_{\textnormal{glo}} = \sum_{i=1}^{N_c} \sum_{j=1}^{l_i} c_j^i \psi_j^i$. 
Subsequently, we express the solution utilizing identical coefficients, but employing local multiscale basis functions, given by $v = \sum_{i=1}^{N_c} \sum_{j=1}^{l_i} c_j^i \psi_{j,\textnormal{ms}}^i \in V_{\textnormal{ms}}$. 
So, by the Galerkin orthogonality, we have

$$
\Vert u-u_{\textnormal{ms}} \Vert_a \le \Vert u-v \Vert_a \le
\Vert u-u_{\textnormal{glo}} \Vert_a + 
\Vert \sum_{i=1}^{N_c} \sum_{j=1}^{l_i} c_j^i
(\psi_j^i-\psi_{j,\textnormal{ms}}^i) \Vert_a.
$$

Recall that the basis functions $\psi_{j,\textnormal{ms}}^i$ 
have supports in $K_{i,k_i}$. So, by \cref{lemma:glolocal_c} and 
\cref{lemma:psi_gloms_reccu},
\begin{align*}
\Vert \sum_{i=1}^{N_c} \sum_{j=1}^{l_i} c_j^i 
(\psi_j^i - \psi_{j,\textnormal{ms}}^i) \Vert_a^2 &\le
C (k+1)^d \sum_{i=1}^{N_c} \Vert \sum_{j=1}^{l_i}
c_j^i (\psi_j^i - \psi_{j,\textnormal{ms}}^i) \Vert_a^2 \\
&\le C (k+1)^d \sum_{i=1}^{N_c} E_i \Vert \sum_{j=1}^{l_i} 
c_j^i \phi_j^i \Vert_s^2 \\
&\le C (k+1)^d E \Vert u_{\textnormal{glo}} \Vert_s^2
\end{align*}
where the last inequality is because $\Vert \sum_{i=1}^{N_c} \sum_{j=1}^{l_i} c_j^i \phi_j^i \Vert_s^2 = \Vert \pi(u_{\textnormal{glo}}) \Vert_s^2$.
By using \cref{lemma:uuglo_esti}, we obtain

$$
\Vert u-u_{\textnormal{ms}} \Vert_a \le C \Lambda_{\Omega^{\epsilon}}^{-\frac{1}{2}}
\Vert \tilde{\kappa}^{-\frac{1}{2}} f \Vert + 
C(k+1)^{\frac{d}{2}} E^{\frac{1}{2}} \Vert u_{\textnormal{glo}}\Vert_s.
$$
This completes the proof for the first part of the theorem.

To proof the second inequality, we need to estimate the 
$s\textnormal{-norm}$ of the global solution $u_{\textnormal{glo}}.$
In particular, 

$$
\Vert u_{\textnormal{glo}} \Vert_s^2 \le \max\{\tilde{\kappa}\}
\Vert u_{\textnormal{glo}} \Vert^2 \le 
C \max\{\tilde{\kappa}\} \Vert u_{\textnormal{glo}} \Vert_a^2.
$$
Since $u_{\textnormal{glo}}$ satisfies 
\eqref{eq:perforated_galerkin_glo}, we have

$$
\Vert u_{\textnormal{glo}} \Vert_a^2 = \int_{\Omega^{\epsilon}} 
f u_{\textnormal{glo}} \le \Vert \tilde{\kappa}^{-\frac{1}{2}} f \Vert 
\Vert u_{\textnormal{glo}} \Vert_s.
$$
Therefore, we have 

$$
\Vert u_{\textnormal{glo}} \Vert_s \le C \max\{\tilde{\kappa}\} 
\Vert \tilde{\kappa}^{-\frac{1}{2}} f \Vert.
$$
To get the second inequality, we need make $C (k+1)^{\frac{d}{2}} E^{\frac{1}{2}}\max\{\tilde{\kappa}\}$ can be bounded.
That is

$$
H^{-2} C (k_i+1)^{\frac{d}{2}} E_i^{\frac{1}{2}} = O(1).
$$
Taking logarithm,

$$
\log(H^{-2}) + \frac{1-k_i}{2} \log\left(
1 + \frac{\Lambda_{K_{i,k_i-1}}^{\frac{1}{2}}}{2D_{K_{i,k_i-1}}^{\frac{1}{2}}} \right) + \frac{d}{2} \log(k_i+1) = O(1)
$$
That is if we take $k_i=O\left(\log(H^{-1})/\log\left(
1 + \Lambda_{K_{i,k_i-1}}^{\frac{1}{2}}\Gamma_{K_{i,k_i-1}}^{-\frac{1}{2}} \right)\right)$ and assume that $\{\chi_i\}_{i=1}^{N_v}$ is a set of bilinear partition of unity, 
then we have

$$
\Vert u-u_{\textnormal{ms}} \Vert_a \le CH \Lambda_{\Omega^{\epsilon}}^{-\frac{1}{2}} \Vert f \Vert.
$$
This completes the proof.
\end{proof}

\subsection{Relaxed Version}
In this subsection, we will prove the convergence when the multiscale basis functions is in relaxed version. 
Similar with constraint version, we will prove relaxed version of multiscale basis functions have a decay property first. 

The same as the constraint version (\cref{lemma:glolocal_c}), 
we also need to estimate the difference between local multiscale basis functions $\psi_{j,\textnormal{ms}}^i$ and global multiscale basis functions $\psi_{j}^i$.
We will prove the exponential decay property depends on the oversampling coarse layers $k_i$ and the eigenvalues ratio.

\begin{lemma} \label{lemma:glolocal_r}
Let $\phi_j^i \in V_{\textnormal{aux}}$ be a given auxiliary function.
Suppose that $\psi_{j,\textnormal{ms}}^i$ is a multiscale basis function obtained in \eqref{eq:r_mini} over the oversampling domain $K_{i,k_i}$ with $k_i\ge 2$ and $\psi_j^i$ is the corresponding global basis function obtained in \eqref{eq:r_mini_global}. Then, the following estimate holds:

$$
\Vert \psi_j^i - \psi_{j,\textnormal{ms}}^i \Vert_a^2 + 
\Vert \pi(\psi_j^i - \psi_{j,\textnormal{ms}}^i) \Vert_s^2 \le 
E_i \left( \Vert \psi_j^i \Vert_a^2 + \Vert \pi(\psi_j^i) \Vert_s^2 \right),
$$
where $E_i=3(1+\Lambda_{\Omega^{\epsilon}\setminus K_{i,k_i-1}}^{-1}) \left(1+(2(1+\Lambda_{K_{i,k_i-1}}^{-\frac{1}{2}}))^{-1}\right)^{1-k_i}$ 
is a factor of exponential decay.
\end{lemma}
\begin{proof}
By the weak form of $\psi_{j,\textnormal{ms}}^i$ and 
$\psi_j^i$ in \eqref{eq:r_variation} and 
\eqref{eq:r_variation_global}, we have

$$
a(\psi_j^i-\psi_{j,\textnormal{ms}}^i, v) + 
s(\pi(\psi_j^i-\psi_{j,\textnormal{ms}}^i), \pi(v)) = 0, \quad
\forall v \in V_0(K_{i,k_i}).
$$
Taking 
$v=w-\psi_{j,\textnormal{ms}}^i$ with $w \in V_0(K_{i,k_i})$ in 
the above relation, we have

$$
\Vert \psi_j^i - \psi_{j,\textnormal{ms}}^i \Vert_a^2 + 
\Vert \pi(\psi_j^i - \psi_{j,\textnormal{ms}}^i) \Vert_s^2 
\le \Vert \psi_j^i - w \Vert_a^2 + 
\Vert \pi(\psi_j^i - w) \Vert_s^2, \quad
\forall w \in V_{0}(K_{i,k_i}).
$$
Let $w=\chi_i^{k_i,k_i-1}\psi_j^i$ in the above relation, 
we have 

\begin{equation}
\Vert \psi_j^i - \psi_{j,\textnormal{ms}}^i \Vert_a^2 + 
\Vert \pi(\psi_j^i - \psi_{j,\textnormal{ms}}^i) \Vert_s^2 
\le \Vert (1-\chi_i^{k_i,k_i-1}) \psi_j^i \Vert_a^2 + 
\Vert \pi((1-\chi_i^{k_i,k_i-1}) \psi_j^i) \Vert_s^2.
\end{equation}
From the \cref{lemma:aux_property} (iv) and (v), we have

$$
\begin{aligned}
&\Vert \psi_j^i - \psi_{j,\textnormal{ms}}^i \Vert_a^2 + 
\Vert \pi(\psi_j^i - \psi_{j,\textnormal{ms}}^i) \Vert_s^2 \\
\le& 3(1+\Lambda_{\Omega^{\epsilon}\setminus K_{i,k_i-1}}^{-1}) 
\left( \Vert \psi_j^i 
\Vert_{a(\Omega^{\epsilon}\setminus K_{i,k_i-1})}^2 + 
\Vert \pi(\psi_j^i)
\Vert_{s(\Omega^{\epsilon} \setminus K_{i,k_i-1})}^2
\right)
\end{aligned}
$$
Next, we will estimate $\Vert \psi_j^i 
\Vert_{a(\Omega^{\epsilon}\setminus K_{i,k_i-1})}^2 + 
\Vert \pi(\psi_j^i)
\Vert_{s(\Omega^{\epsilon} \setminus K_{i,k_i-1})}^2.$
We will show that this term can be bounded by the term: 
$\Vert \psi_j^i 
\Vert_{a(\Omega^{\epsilon}\setminus K_{i,k_i-2})}^2 + 
\Vert \pi(\psi_j^i)
\Vert_{s(\Omega^{\epsilon} \setminus K_{i,k_i-2})}^2.$
This recursive property is crucial in our convergence estimate.

Choosing test function $v=(1-\chi_i^{k_i-1,k_i-2})\psi_j^i$ in \eqref{eq:r_variation_global}, we have

\begin{equation}
\begin{aligned}
& a(\psi_j^i, (1-\chi_i^{k_i-1,k_i-2})\psi_j^i) + 
s(\pi(\psi_j^i), \pi((1-\chi_i^{k_i-1,k_i-2})\psi_j^i)) \\
=& s(\phi_j^i, \pi((1-\chi_i^{k_i-1,k_i-2})\psi_j^i)) = 0
\end{aligned}
\label{eq:psipi_k1k2}
\end{equation}
where the last equality follows from the facts that 
$(1-\chi_i^{k_i-1,k_i-2}) \psi_j^i$ and $\phi_j^i$ have disjoint support.
Note that 

$$
\begin{aligned}
& a(\psi_j^i, (1-\chi_i^{k_i-1,k_i-2})\psi_j^i) \\
=& \int_{\Omega^{\epsilon}\setminus K_{i,k_i-2}} 
\nabla \psi_j^i \cdot \nabla((1-\chi_i^{k_i-1,k_i-2})
\psi_j^i) \\
=& 
\int_{\Omega^{\epsilon} \setminus K_{i,k_i-2}} 
(1-\chi_i^{k_i-1,k_i-2}) |\nabla \psi_j^i|^2 - 
\int_{\Omega^{\epsilon} \setminus K_{i,k_i-2}} 
\psi_j^i \nabla \chi_i^{k_i-1,k_i-2} \cdot \nabla \psi_j^i.
\end{aligned}
$$
Consequently, we have

\begin{equation}
\begin{aligned}
\Vert \psi_j^i 
\Vert_{a(\Omega^{\epsilon} \setminus K_{i,k_i-1})}^2 \le& 
\int_{\Omega^{\epsilon} \setminus K_{i,k_i-2}} 
(1-\chi_i^{k_i-1,k_i-2}) |\nabla \psi_j^i|^2 \\
=& a(\psi_j^i, (1-\chi_i^{k_i-1,k_i-2})\psi_j^i) + 
\int_{\Omega^{\epsilon} \setminus K_{i,k_i-2}} 
\psi_j^i \nabla \chi_i^{k_i-1,k_i-2} \cdot \nabla \psi_j^i.
\\ \le& a(\psi_j^i, (1-\chi_i^{k_i-1,k_i-2})\psi_j^i) \\
&+ \Vert \psi_j^i \Vert_{a(K_{i,k_i-1}\setminus K_{i,k_i-2})}
\Vert \psi_j^i \Vert_{s(K_{i,k_i-1}\setminus K_{i,k_i-2})}.
\end{aligned}
\label{eq:psi_omega_ik2}
\end{equation}
Next, from the definition of cutoff function, we can conclude,

$$
\begin{aligned}
& s\left( \pi(\psi_j^i), \pi \left( (1-\chi_i^{k_i-1,k_i-2})
\psi_j^i \right) \right) \\
=& \Vert \pi(\psi_j^i) 
\Vert_{s(\Omega^{\epsilon} \setminus K_{i,k_i-1})} + 
\int_{K_{i,k_i-1} \setminus K_{i,k_i-2}} \tilde{\kappa}
\pi(\psi_j^i) \pi((1-\chi_i^{k_i-1,k_i-2}) \psi_j^i),
\end{aligned}
$$
so we have

\begin{equation}
\begin{aligned}
& \Vert \pi(\psi_j^i) 
\Vert_{s(\Omega^{\epsilon} \setminus K_{i,k_i-1})}^2 \\
=& s( \pi(\psi_j^i), \pi((1-\chi_i^{k_i-1,k_i-2}) \psi_j^i)) - 
\int_{K_{i,k_i-1} \setminus K_{i,k_i-2}} \tilde{\kappa} 
\pi(\psi_j^i) \pi((1-\chi_i^{k_i-1,k_i-2}) \psi_j^i) \\
\le& s( \pi(\psi_j^i), \pi((1-\chi_i^{k_i-1,k_i-2}) 
\psi_j^i )) + 
\Vert \psi_j^i \Vert_{s(K_{i,k_i-1} \setminus K_{i,k_i-2})}
\Vert \pi(\psi_j^i) \Vert_{s(K_{i,k_i-1} \setminus K_{i,k_i-2})}.
\end{aligned}
\label{eq:pipsi_omega_dele_ik1}
\end{equation}
Finally, summing \eqref{eq:psi_omega_ik2} and 
\eqref{eq:pipsi_omega_dele_ik1} and using \eqref{eq:psipi_k1k2},
we have

\begin{equation}
\begin{aligned}
&\Vert \psi_j^i \Vert_{a(\Omega^{\epsilon} \setminus K_{i,k_i-1})}^2 + 
\Vert \pi(\psi_j^i) \Vert_{s(\Omega^{\epsilon} \setminus K_{i,k_i-1})}^2 \\
\le & \Vert \psi_j^i \Vert_{s(K_{i,k_i-1} \setminus K_{i,k_i-2})} 
\left(
\Vert \psi_j^i \Vert_{a(K_{i,k_i-1} \setminus K_{i,k_i-2})} +
\Vert \pi(\psi_j^i) \Vert_{s(K_{i,k_i-1} \setminus K_{i,k_i-2})}
\right) \\
\le & 2(1+\Lambda_{K_{i,k_i-1} \setminus K_{i,k_i-2}}^{-\frac{1}{2}}) \left(
\Vert \psi_j^i \Vert_{a(K_{i,k_i-1} \setminus K_{i,k_i-2})}^2 +
\Vert \pi(\psi_j^i) \Vert_{s(K_{i,k_i-1} \setminus K_{i,k_i-2})}^2
\right)
\end{aligned}
\label{ineq:r_psirecu_k1k2}
\end{equation}
where the last inequality follows from \cref{lemma:aux_property} (iii). 
By using this, we have

$$
\begin{aligned}
& \Vert \psi_j^i \Vert_{a(\Omega^{\epsilon} \setminus K_{i,k_i-2})}^2 + \Vert \pi(\psi_j^i) \Vert_{s(\Omega^{\epsilon} \setminus K_{i,k_i-2})}^2 \\
=& \Vert \psi_j^i \Vert_{a(\Omega^{\epsilon} \setminus K_{i,k_i-1})}^2 + 
\Vert \pi(\psi_j^i) \Vert_{s(\Omega^{\epsilon} \setminus K_{i,k_i-1})}^2 \\
&+ \Vert \psi_j^i \Vert_{a(K_{i,k_i-1} \setminus K_{i,k_i-2})}^2 + 
\Vert \pi(\psi_j^i) \Vert_{s(K_{i,k_i-1} \setminus K_{i,k_i-2})}^2 \\
\ge& \left( 1+(2(1+\Lambda_{K_{i,k_i-1} \setminus K_{i,k_i-2}}^{-\frac{1}{2}}))^{-1} \right) 
(\Vert \psi_j^i \Vert_{a(\Omega^{\epsilon} \setminus K_{i,k_i-1})}^2 + \Vert \pi(\psi_j^i) \Vert_{s(\Omega^{\epsilon} \setminus K_{i,k_i-1})}^2)
\end{aligned}
$$
where we used \eqref{ineq:r_psirecu_k1k2} in the last inequality. Using the inequality recursively, we have

$$
\begin{aligned}
&\Vert \psi_j^i 
\Vert_{a(\Omega^{\epsilon} \setminus K_{i,k_i-1})}^2 
+ \Vert \pi(\psi_j^i) \Vert_{s(\Omega^{\epsilon} \setminus K_{i,k_i-1})}^2 \\
\le& \left( 1+(2(1+\Lambda_{K_{i,k_i-1}}^{-\frac{1}{2}}))^{-1} \right)^{1-k_i}
\left(\Vert \psi_j^i \Vert_a^2 + \Vert \pi(\psi_j^i) \Vert_s^2\right).
\end{aligned}
$$
This completes the proof.
\end{proof}

Different with the constraint version \cref{lemma:glolocal_r}, the above lemma has improved the convergence rate, because the error bound is independent of the constant $D$.
We need the following lemma to prove the convergence, see \cite{chung2018fast}.

\begin{lemma} \label{lemma:psi_gloms_recu}
With the same notation in \cref{lemma:glolocal_r}, we have

$$
\begin{aligned}
& \Vert \sum_{i=1}^{N_c} \sum_{j=1}^{l_i} c_j^i (\psi_j^i - \psi_{j,\text{ms}}^i) \Vert_a^2 + 
\Vert \sum_{i=1}^{N_c} \sum_{j=1}^{l_i} c_j^i \pi(\psi_j^i - \psi_{j,\text{ms}}^i) \Vert_s^2 \\
\le & C (1+\Lambda_{\Omega^{\epsilon}}^{-1}) (k+1)^d 
\sum_{i=1}^{N_c} \left( \Vert \sum_{j=1}^{l_i} c_j^i (\psi_j^i - \psi_{j,\text{ms}}^i) \Vert_a^2 + 
\Vert \sum_{j=1}^{l_i} c_j^i \pi(\psi_j^i - \psi_{j,\text{ms}}^i) \Vert_s^2 \right).    
\end{aligned}
$$
\end{lemma}

\begin{theorem} \label{thm:convergence_r}
Let $u$ be the solution of \eqref{eq:perforated_variation}  and 
$u_{\textnormal{ms}}$ be the solution of \eqref{eq:perforated_galerkin_ms}. 
Define $E=\max_{1\le i \le N_c}E_i$.
We have 

$$
\Vert u - u_{\textnormal{ms}} \Vert \le C \Lambda_{\Omega^{\epsilon}}^{-\frac{1}{2}}
\Vert \tilde{\kappa}^{-\frac{1}{2}} f \Vert + 
C (k+1)^{\frac{d}{2}}(1+\Lambda_{\Omega^{\epsilon}}^{-1})^{\frac{1}{2}} (1+D_{\Omega^{\epsilon}})^{\frac{1}{2}} E^{\frac{1}{2}} \Vert u_{\textnormal{glo}} \Vert_s,
$$
where $u_{\textnormal{glo}}$ is the solution of 
\eqref{eq:perforated_galerkin_glo}. 
Moreover, if the oversampling parameter
$k$ is sufficiently large 
and $\{\chi_i\}_{i=1}^{N_v}$ is a set of bilinear 
partition of unity, we have 

$$
\Vert u - u_{\textnormal{ms}} \Vert_a \le C H \Lambda_{\Omega^{\epsilon}}^{-\frac{1}{2}} \Vert f \Vert.
$$
\end{theorem}
\begin{proof}
The proof follows the same procedure as the proof of \cref{thm:convergence_c}. We write 
$u_{\textnormal{glo}} = \sum_{i=1}^{N_c}\sum_{j=1}^{l_i} c_{ij}\psi_j^i$ and 
define $v \coloneqq \sum_{i=1}^{N_c}\sum_{j=1}^{l_i} c_{ij}\psi_{j,\textnormal{ms}}^i$. 
By the \cref{lemma:glolocal_r} and \cref{lemma:psi_gloms_recu}, we have

$$
\begin{aligned}
&\Vert u_{\textnormal{glo}} - v \Vert_a^2 =
\Big\Vert \sum_{j=1}^{N_c} \sum_{j=1}^{l_i} 
c_{ij} (\psi_j^i - \psi_{j,\textnormal{ms}}^i)\Big\Vert_a^2 \\ 
\le& C(1+\Lambda_{\Omega^{\epsilon}}^{-1}) (k+1)^d \sum_{i=1}^{N_c} 
\left(
\Big\Vert \sum_{j=1}^{l_i} c_{ij} (\psi_j^i - \psi_{j,\textnormal{ms}}^i)
\Big\Vert_a^2 + \Big\Vert 
\sum_{j=1}^{l_i} c_{ij} \pi(\psi_j^i  - \psi_{j,\textnormal{ms}}^i )\Big\Vert_s^2
\right) \\
\le& C (1+\Lambda_{\Omega^{\epsilon}}^{-1}) (k+1)^d \sum_{i=1}^{N_c} E_i \sum_{j=1}^{l_i} (c_{ij})^2 
\left(
\Vert \psi_j^i \Vert_a^2 + \Vert \pi(\psi_j^i)  \Vert_s^2
\right).
\end{aligned}
$$
Choosing the test function $v=\psi_j^i$ in \eqref{eq:r_variation_global}, we obtain that 
$\Vert \psi_j^i \Vert_a^2 + \Vert \pi(\psi_j^i) \Vert_s^2 \le \Vert \phi_j^i \Vert_s^2=1.$
Therefore, 

$$
\begin{aligned}
\Vert u_{\textnormal{glo}} - v \Vert_a^2 
\le& C(k+1)^d(1+\Lambda_{\Omega^{\epsilon}}^{-1}) (1+D_{\Omega^{\epsilon}}) E \sum_{i=1}^{N_c} 
\sum_{j=1}^{l_i} (c_{ij})^2 \Vert \phi_j^i\Vert_s^2 \\
=& C(k+1)^d(1+\Lambda_{\Omega^{\epsilon}}^{-1}) (1+D_{\Omega^{\epsilon}}) E \sum_{i=1}^{N_c} \sum_{j=1}^{l_i} (c_{ij})^2.
\end{aligned}
$$
Next, we will estimate $\sum_{i=1}^{N_c} \sum_{j=1}^{l_i} (c_{ij})^2.$ 
Note that $\pi(u_{\textnormal{glo}})=\sum_{i=1}^{N_c} \sum_{j=1}^{l_i} c_{ij}\pi(\psi_j^i).$ Using the variational formulation \eqref{eq:r_variation_global}, we obtain

$$
\begin{aligned}
b_{lk} &\coloneqq s(\phi_k^l, \pi(u_{\textnormal{glo}})) \\
&= \sum_{i=1}^{N_c} \sum_{j=1}^{l_i} c_{ij} s(\phi_k^l, \pi(\psi_j^i)) \\
&= \sum_{i=1}^{N_c} \sum_{j=1}^{l_i} c_{ij}
\left(
a(\psi_k^l, \psi_j^i) + s(\pi(\psi_k^l), \pi(\psi_j^i))
\right)
\end{aligned}
$$
If we denote $a_{ij,lk} \coloneqq a(\psi_k^l, \psi_j^i) + s(\pi(\psi_k^l), \pi(\psi_j^i))  \in \mathbb{R}^{\mathcal{N}\times\mathcal{N}},
\vec{\boldsymbol{b}} = (b_{lk}) \in \mathbb{R}^{\mathcal{N}}$ 
and $\vec{\boldsymbol{c}} = (c_{ij}) \in \mathbb{R}^{\mathcal{N}}$ with 
$\mathcal{N} \coloneqq \sum_{i=1}^{N_c} l_i$, then we have

$$
\vec{\boldsymbol{b}} = A\vec{\boldsymbol{c}} \quad \textnormal{and} \quad
\Vert \vec{\boldsymbol{c}} \Vert_2 \le \Vert A^{-1} \Vert_2 
\Vert \vec{\boldsymbol{b}} \Vert_2,
$$
where $A \coloneqq (a_{ij,lk}) \in 
\mathbb{R}^{\mathcal{N} \times \mathcal{N}}$ and $\Vert \cdot \Vert_2$ 
denotes the standard Euclidean norm for vectors in 
$\mathbb{R}^{\mathcal{N}}$ and its induced matrix norm in 
$\mathbb{R}^{\mathcal{N} \times \mathcal{N}}$. 
By the definition of $\pi : V \rightarrow V_{\textnormal{aux}}$, we have

$$
\pi(u_{\textnormal{glo}}) = \pi(\pi(u_{\textnormal{glo}})) = 
\sum_{i=1}^{N} \sum_{j=1}^{l_i} 
s(\pi(u_{\textnormal{glo}}), \phi_j^i)\phi_j^i = 
\sum_{i=1}^{N} \sum_{j=1}^{l_i} b_{ij} \phi_j^i.
$$
Thus, we have 
$\Vert \vec{\boldsymbol{b}} \Vert_2 = \Vert \pi(u_{\textnormal{glo}}) \Vert_s$. 
We define $\phi \coloneqq \sum_{i=1}^{N_c} \sum_{j=1}^{l_i} c_{ij}\phi_j^i$. 
Note that $\Vert \phi \Vert_s = \Vert \vec{\boldsymbol{c}} \Vert_2.$ 
Consequently, by \cref{lemma:vaux_v}, there exists a function $z \in V$ 
such that $\pi(z) = \phi$ and $\Vert z \Vert_a^2 \le D_{\Omega^{\epsilon}} \Vert\phi\Vert_s^2$.
Since the global multiscale basis $\psi_j^i$ satisfies \eqref{eq:r_variation_global} 
and $u_{\textnormal{glo}}$ is a linear combination of $\psi_j^i$'s, we have 

\begin{equation}
a(u_{\textnormal{glo}}, v) + s(\pi(u_{\textnormal{glo}}), \pi(v)) = s(\phi, \pi(v))
\textnormal{ for all } v \in V.
\label{eq:glo_varia_r}
\end{equation}
Picking $v=z$ in \eqref{eq:glo_varia_r}, we arrive at 

$$
\begin{aligned}
\Vert \phi \Vert_s^2 
& = a(u_{\textnormal{glo}}, z) + s(\pi(u_{\textnormal{glo}}), \pi(z)) \\
& \le \Vert u_{\textnormal{glo}} \Vert_a \cdot D_{\Omega^{\epsilon}}^{\frac{1}{2}} \Vert \phi \Vert_s + 
\Vert \pi(u_{\textnormal{glo}}) \Vert_s \cdot \Vert \phi \Vert_s \\
& \le (1+D_{\Omega^{\epsilon}})^{\frac{1}{2}} \Vert \phi \Vert_s 
\left(
\Vert u_{\textnormal{glo}} \Vert_a^2 + \Vert \pi (u_{\textnormal{glo})} \Vert_s^2
\right)^{\frac{1}{2}}.
\end{aligned}
$$
Therefore, we have

$$
\Vert \vec{\boldsymbol{c}} \Vert_2^2 = \Vert \phi \Vert_s^2 \le 
(1+D_{\Omega^{\epsilon}}) \left(
\Vert u_{\textnormal{glo}} \Vert_a^2 + \Vert \pi(u_{\textnormal{glo}}) \Vert_s^2
\right) = (1+D_{\Omega^{\epsilon}}) \vec{\boldsymbol{c}}^T A \vec{\boldsymbol{c}}.
$$
From the above, we see that the largest eigenvalue of $A^{-1}$ is bounded 
by $(1+D_{\Omega^{\epsilon}})$ and we have the following estimate 

$$
\Vert \vec{\boldsymbol{c}} \Vert_2 \le (1+D_{\Omega^{\epsilon}}) \Vert \vec{\boldsymbol{b}} \Vert_2
\le (1+D_{\Omega^{\epsilon}}) \Vert u_{\textnormal{glo}} \Vert_s.
$$
As a result, we have

$$
\Vert u_{\textnormal{glo}} - v \Vert_a^2 \le 
C(k+1)^d(1+\Lambda_{\Omega^{\epsilon}}^{-1}) (1+D_{\Omega^{\epsilon}}) E
(1+D_{\Omega^{\epsilon}}) \Vert u_{\textnormal{glo}} \Vert_s^2.
$$
The rest of the proofs follows from \cref{thm:convergence_c}.
\end{proof}

Finally, we remark that all of the main theorem has been proved.

\section{Numerical results} \label{sec:numericalresults}
In this section, we present two examples using the framework presented in Section \ref{sec:cembasis} for equation \eqref{eq:pde_perforated}. 
The goal is to demonstrate the accuracy of our proposed multiscale model reduction method and establish its linear decrease with the coarse mesh size $H$ when the oversampling layers $m$ are appropriately chosen.

\begin{figure}[htbp]
\centering
\includegraphics[scale=0.4]{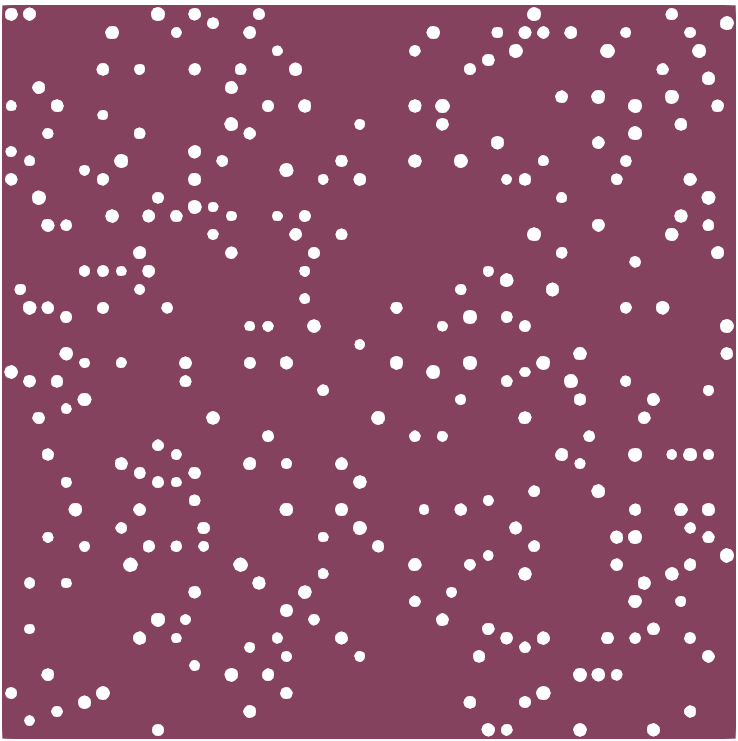}
\includegraphics[scale=0.4]{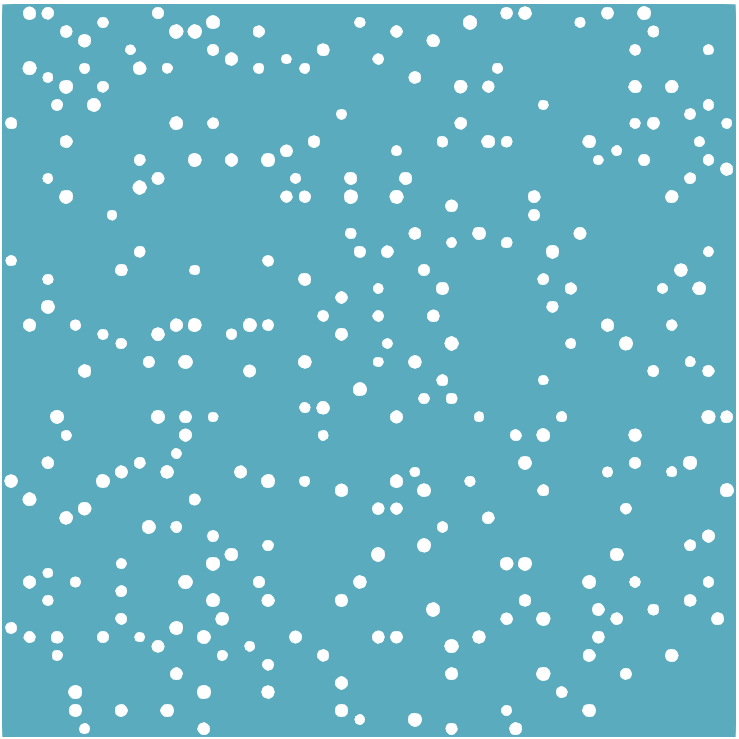}
\caption{Two heterogenous perforated media used in the simulations. The left media will be used for case 1, while the others will be used for case 2.}
\label{pic:two_perforated_domain}
\end{figure}

In all the examples, we set $\Omega = (0,1) \times (0,1)$, and test both constraint version \eqref{eq:c_mini} and relaxed version \eqref{eq:r_mini}. 
In each case, we will use different media and source term.
There are two different perforated domains that depicted in \autoref{pic:two_perforated_domain}, the left will be test in case 1, while the others in case 2. 
In this paper, we assume the perforations $\mathcal{B}^{\epsilon}$ are all circular. 

In the following experiments, we define some numerical solution errors, 

$$
e_{L^2} \coloneqq \frac{\Vert u_h-u_{\textnormal{ms}} \Vert}{\Vert u_h \Vert} , \quad
e_{H^1} \coloneqq \frac{\Vert u_h-u_{\textnormal{ms}} \Vert_{a}}{\Vert u_h \Vert_{a}}.
$$
where $u_h$ is the fine-grid first order FEM solution. 
In here, we have to note that there are only 3 multiscale basis functions be used in each local domain for every cases. 

The computational domains and meshes were constructed using the GMSH software \cite{geuzaine2009gmsh}. In each case, fine-grid triangles generated by GMSH had a diameter of $1/200$. To visualize the numerical results, the ParaView software \cite{ahrens200536} was employed.

\subsection{Case 1}
In this instance, we employ the perforated domain depicted on the left side of \autoref{pic:two_perforated_domain}. The source term is illustrated in \autoref{pic:case1_f}. Specifically, the source term exhibits only two distinct values within the domain, where we designate four subdomains with a value of 1 (depicted in red) and assign a value of 0 to the remaining regions (depicted in blue).

\begin{figure}[htbp]
\centering
\includegraphics[scale=0.45]{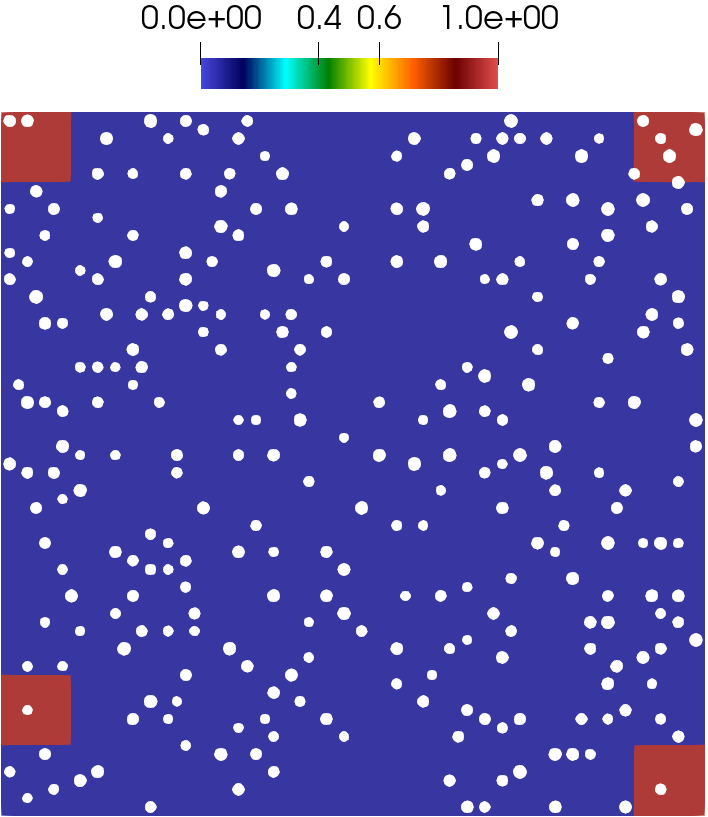}
\caption{Source term of case 1.}
\label{pic:case1_f}
\end{figure}

\begin{table}[htbp]
\centering
\begin{tabular}{|c|c|cc|cc|}
\hline
\multirow{2}*{$H$} & \multirow{2}*{$m$} & 
\multicolumn{2}{|c|}{Constraint Version} & 
\multicolumn{2}{|c|}{Relaxed Version}
\\ \cline{3-6}
& & $e_{L^2}$ & $e_{H^1}$
& $e_{L^2}$ & $e_{H^1}$ \\ \hline
$1/10$ & 2 & 4.80e-02 & 2.04e-01 & 4.56e-02 & 2.02e-01
\\
$1/20$ & 3 & 5.72e-03 & 5.05e-02 & 5.56e-03 & 4.98e-02
\\ 
$1/40$ & 4 & 3.66e-04 & 6.35e-03 & 3.50e-04 & 6.11e-03
\\ \hline
\end{tabular}
\caption{Numerical errors of CEM-GMsFEM using 3 basis functions in each oversampling domain $K_{i,m_i}$ with different coarse mesh size $H.$ The source and domain are depicted in \autoref{pic:case1_f}.}
\label{tab:case1_error}
\end{table}

\begin{figure}[htbp]
\centering
\includegraphics[scale=0.36]{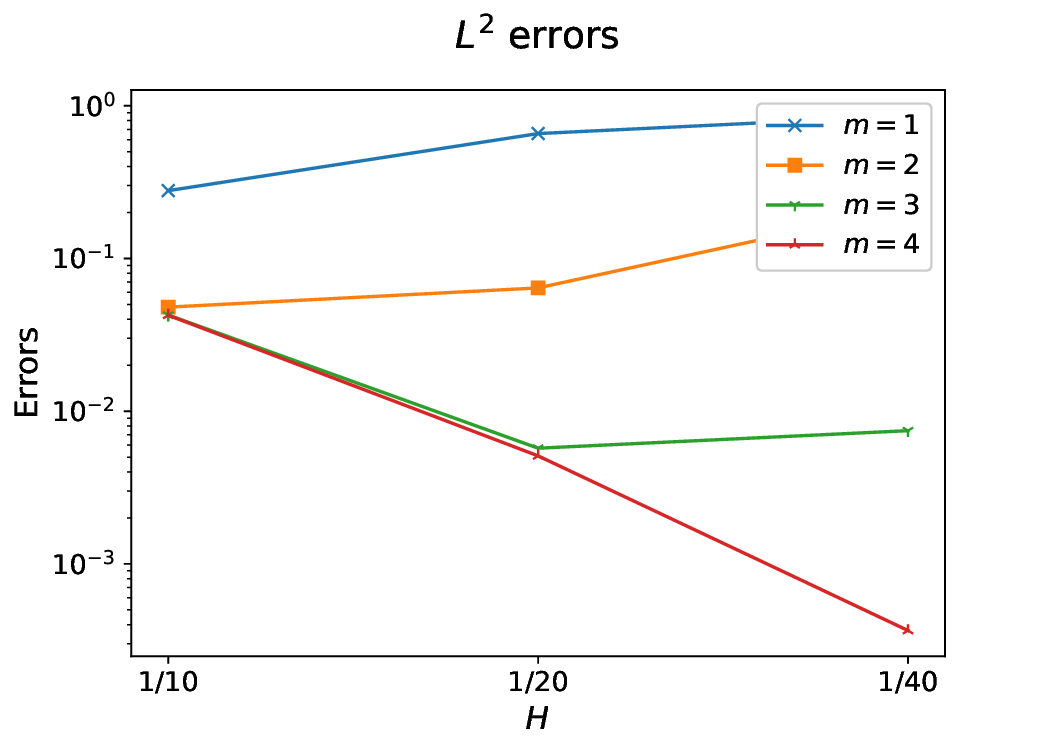}
\includegraphics[scale=0.36]{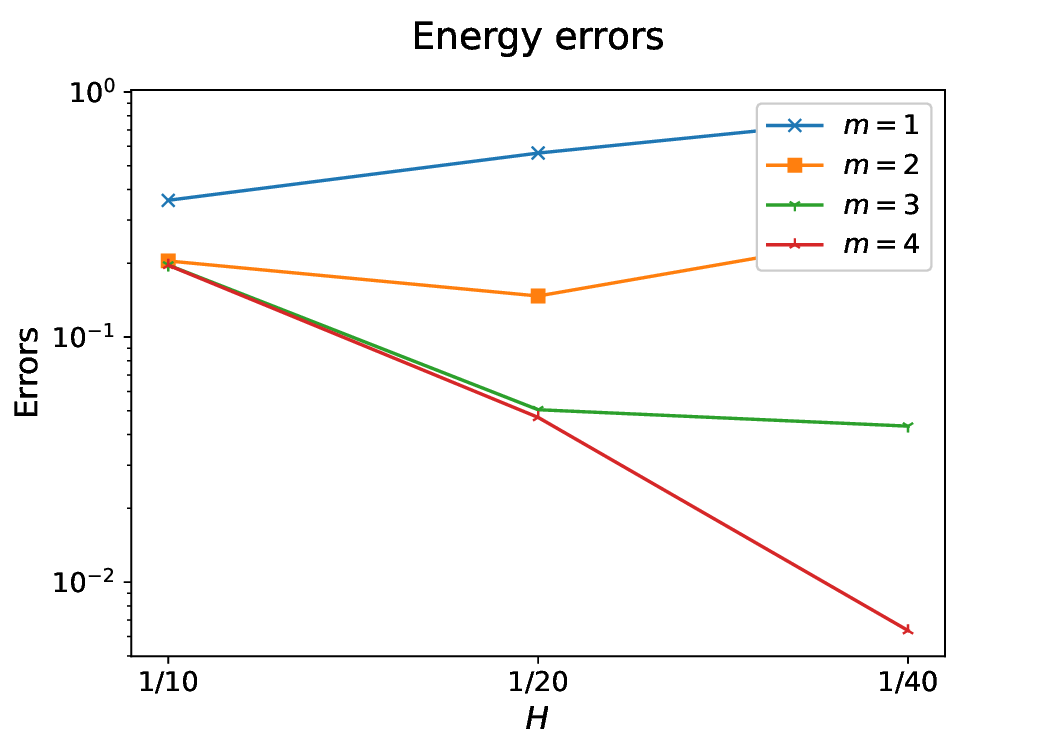}
\caption{Relative L2 error (Left) and Relative h1 error (Right) of Constraint version of CEM-GMsFEM. In here, we consider case 1, the source term and domain are depicted in \autoref{pic:case1_f}.}
\label{pic:case1_error_c}
\end{figure}

\begin{figure}[htbp]
\centering
\includegraphics[scale=0.36]{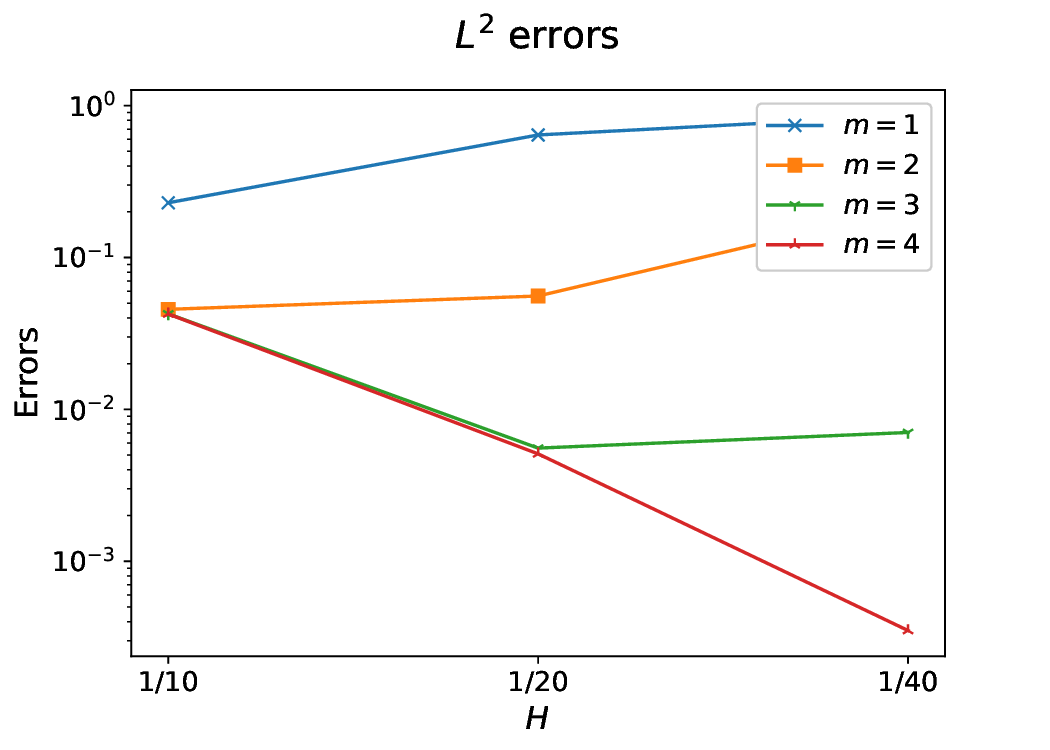}
\includegraphics[scale=0.36]{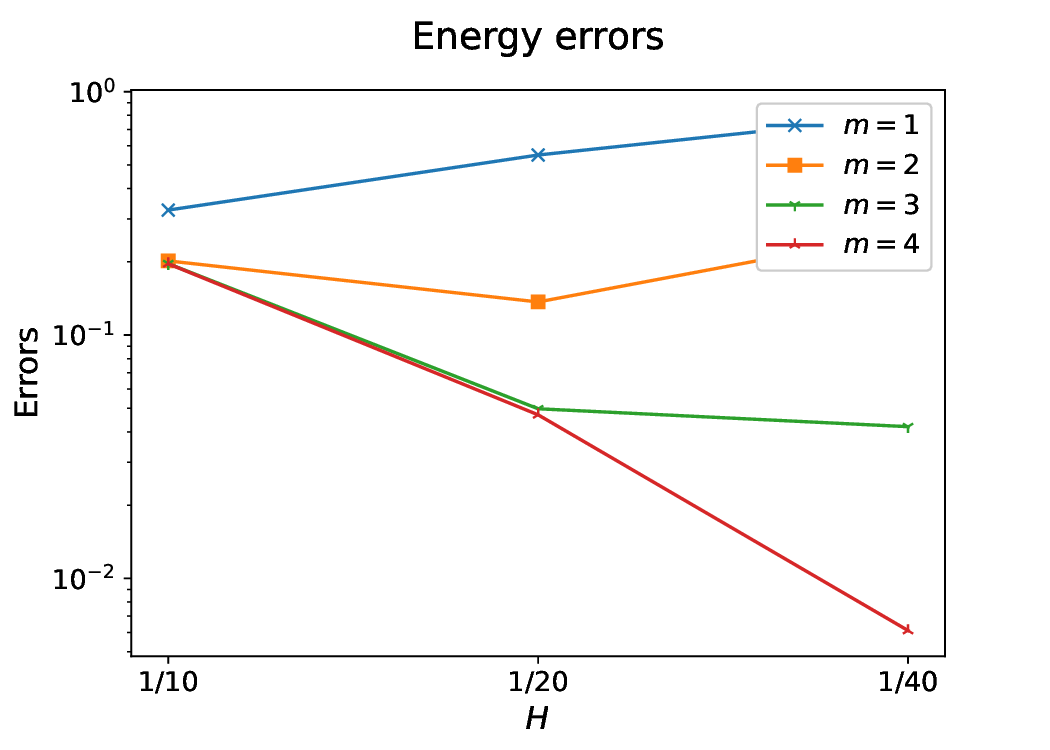}
\caption{Relative L2 error (Left) and Relative h1 error (Right) of Relaxed version of CEM-GMsFEM. In here, we consider case 1, the source term and domain are depicted in \autoref{pic:case1_f}.}
\label{pic:case1_error_r}
\end{figure}

\begin{figure}[htbp]
\centering
\includegraphics[scale=0.33]{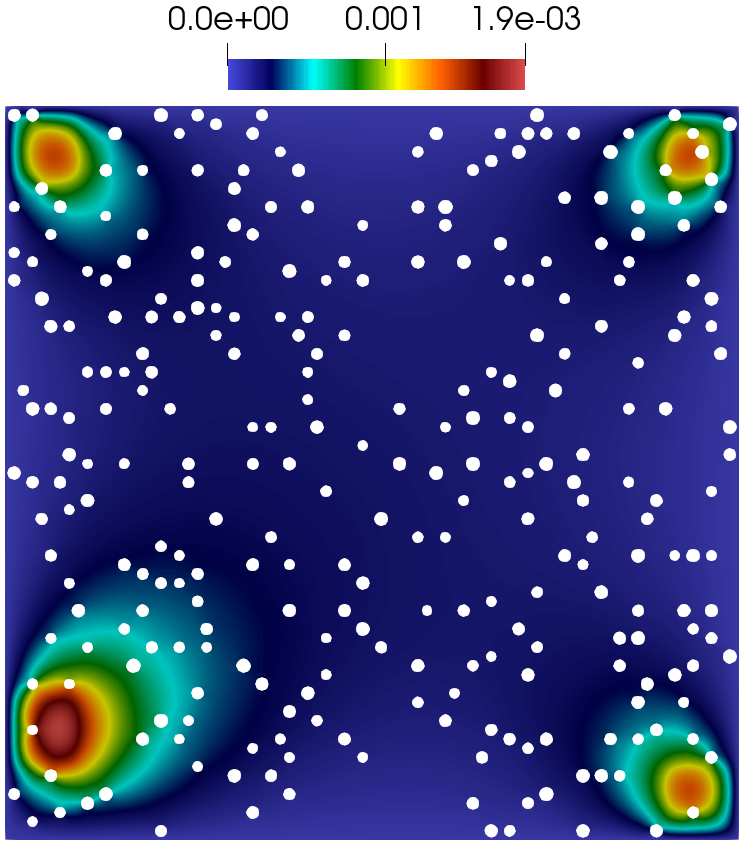}
\includegraphics[scale=0.33]{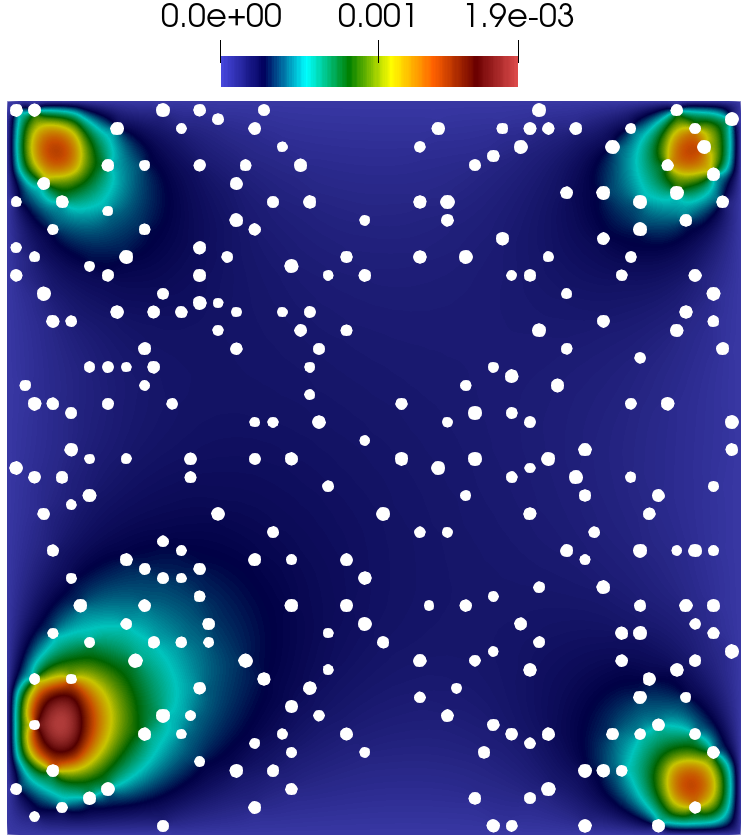}
\includegraphics[scale=0.33]{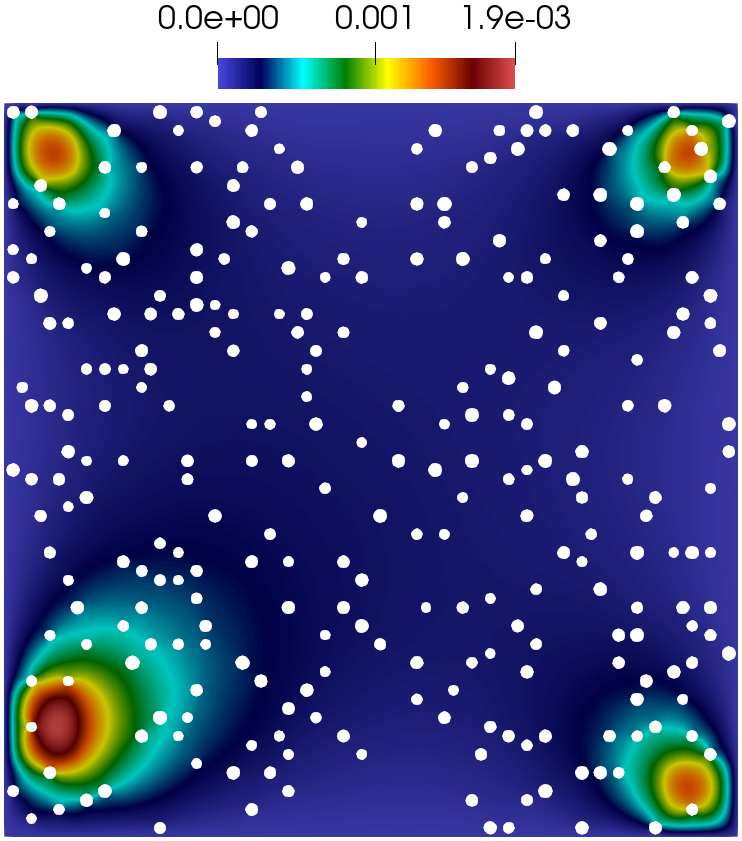}
\caption{Solutions: Reference solution (top left). Using $H=1/40$ and oversampling $3$ coarse layer, constraint version (top right), relaxed version (bottom). In here, we consider case 1, the source term and domain are depicted in \autoref{pic:case1_f}.}
\label{pic:case1_cemuh}
\end{figure}

In \autoref{tab:case1_error}, we detail the errors resulting from the refinement of the coarse mesh size alongside appropriately chosen oversampling layers. 
Both the constraint and relaxed versions of the multiscale solution demonstrate convergence towards the fine-grid solution, as evidenced by diminishing relative L2 and h1 errors with decreasing $H$. 
We further explore the influence of varying oversampling layers while maintaining a constant coarse mesh size $H$.
Figures \ref{pic:case1_error_c} (constraint version) and \ref{pic:case1_error_r} (relaxed version) illustrate the errors of the CEM-GMsFEM for varying coarse sizes and oversampling coarse layers. 
Significantly, these figures highlight the inherent advantage of a linear reduction in error with an increase in coarse mesh size.
Additionally, we visually represent the reference solution, constraint oversampling 3 coarse layers (all) diagonal numerical solutions from \autoref{pic:case1_error_c} and \ref{pic:case1_error_r} in \autoref{pic:case1_cemuh}. 
This provides a more vivid example showcasing the diminishing advantage of our method as $H$ decreases.

\subsection{Case 2}
In this specific scenario, we make use of heterogeneous media showcased on the right side of \autoref{pic:two_perforated_domain}. The source term is visualized in \autoref{pic:case2_f}. Here, the source term is characterized by three distinct values distributed throughout the domain: -1 (depicted in blue), 0 (in green), and 1 (highlighted in red) within specified regions.

\begin{figure}[htbp]
\centering
\includegraphics[scale=0.45]{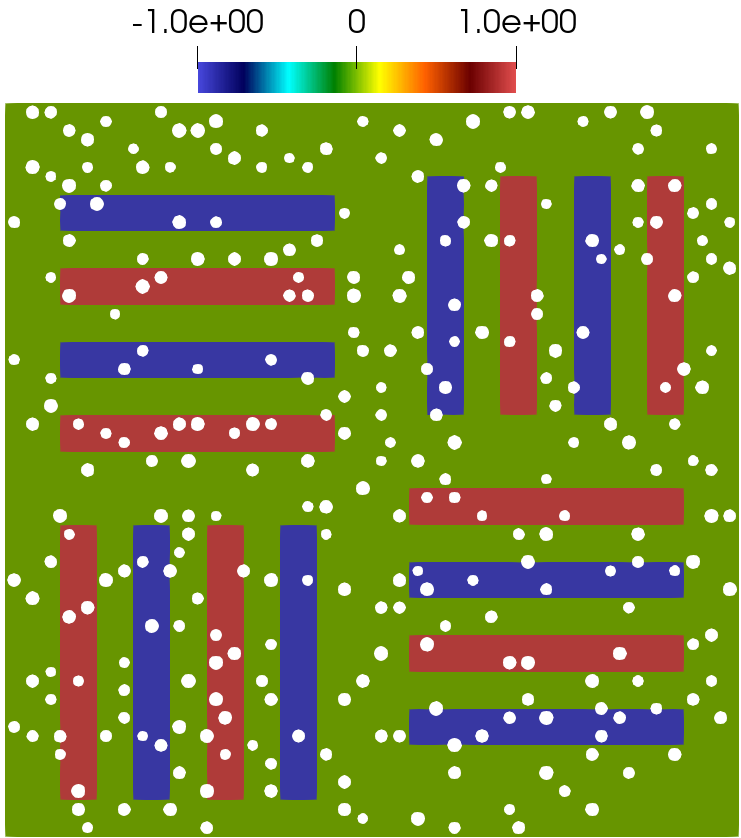}
\caption{Source term of case 2. There are only three values in the source term, $-1$ (blue), $0$ (green) and $1$ (red).}
\label{pic:case2_f}
\end{figure}

\begin{table}[htbp]
\centering
\begin{tabular}{|c|c|cc|cc|}
\hline
\multirow{2}*{$H$} & \multirow{2}*{$m$} & 
\multicolumn{2}{|c|}{Constraint Version} & 
\multicolumn{2}{|c|}{Relaxed Version}
\\ \cline{3-6}
& & $e_{L^2}$ & $e_{H^1}$
& $e_{L^2}$ & $e_{H^1}$ \\ \hline
$1/10$ & 2 & 2.38e-02 & 3.60e-02 & 1.40e-02 & 2.66e-02
\\
$1/20$ & 3 & 4.77e-03 & 1.58e-02 & 3.85e-03 & 1.40e-02
\\
$1/40$ & 4 & 4.98e-04 & 5.21e-03 & 4.68e-04 & 5.02e-03
\\ \hline
\end{tabular}
\caption{Numerical errors of CEM-GMsFEM using 3 basis functions in each oversampling domain $K_{i,m_i}$ with different coarse mesh size $H.$ In here, we consider case 2, the source term and domain are depicted in \autoref{pic:case2_f}.}
\label{tab:case2_error}
\end{table}

\begin{figure}[htbp]
\centering
\includegraphics[scale=0.36]{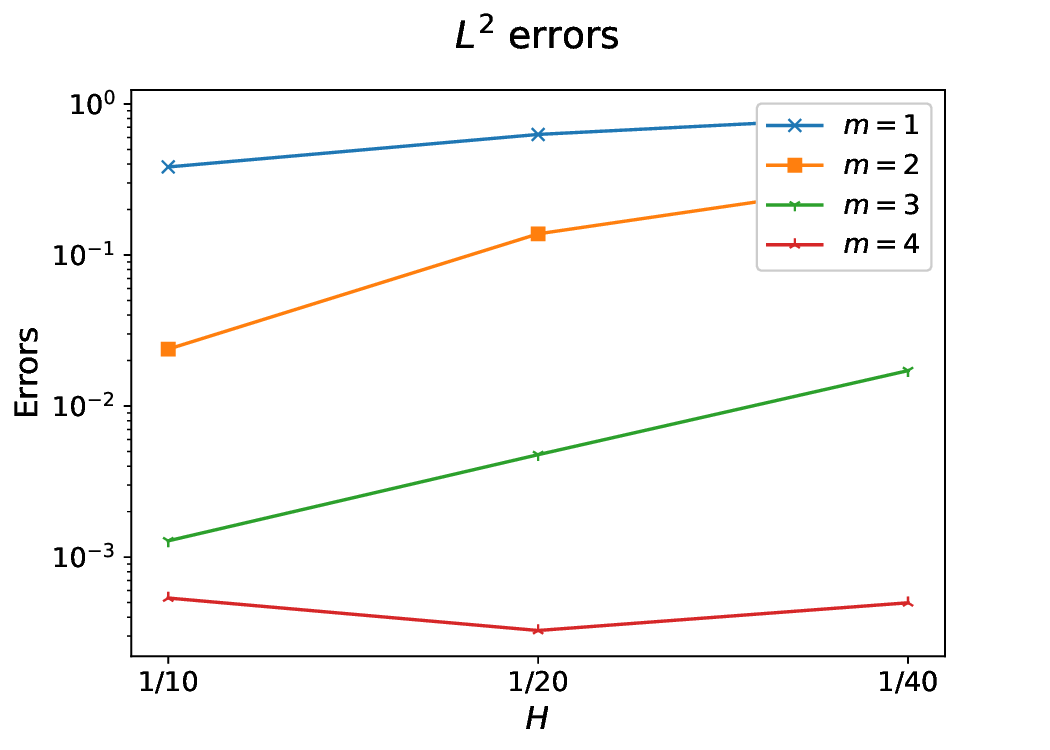}
\includegraphics[scale=0.36]{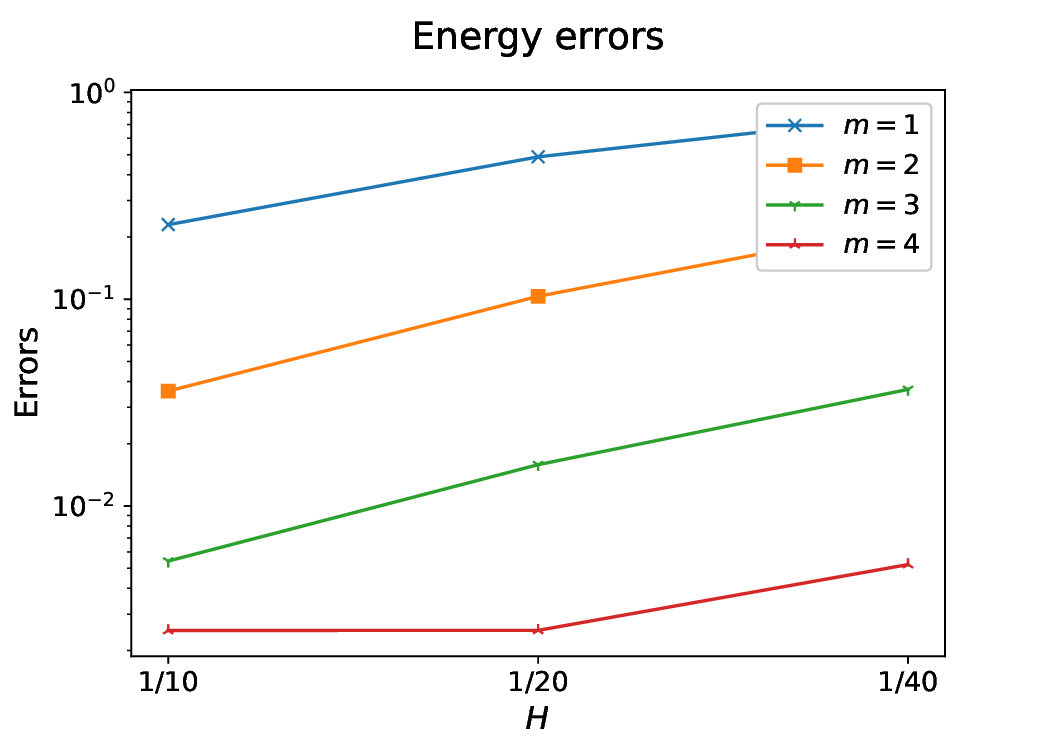}
\caption{Relative L2 error (Left) and Relative h1 error (Right) of Constraint version of CEM-GMsFEM. In here, we consider case 2, the source term and domain are depicted in \autoref{pic:case2_f}.}
\label{pic:case2_error_c}
\end{figure}

\begin{figure}[htbp]
\centering
\includegraphics[scale=0.36]{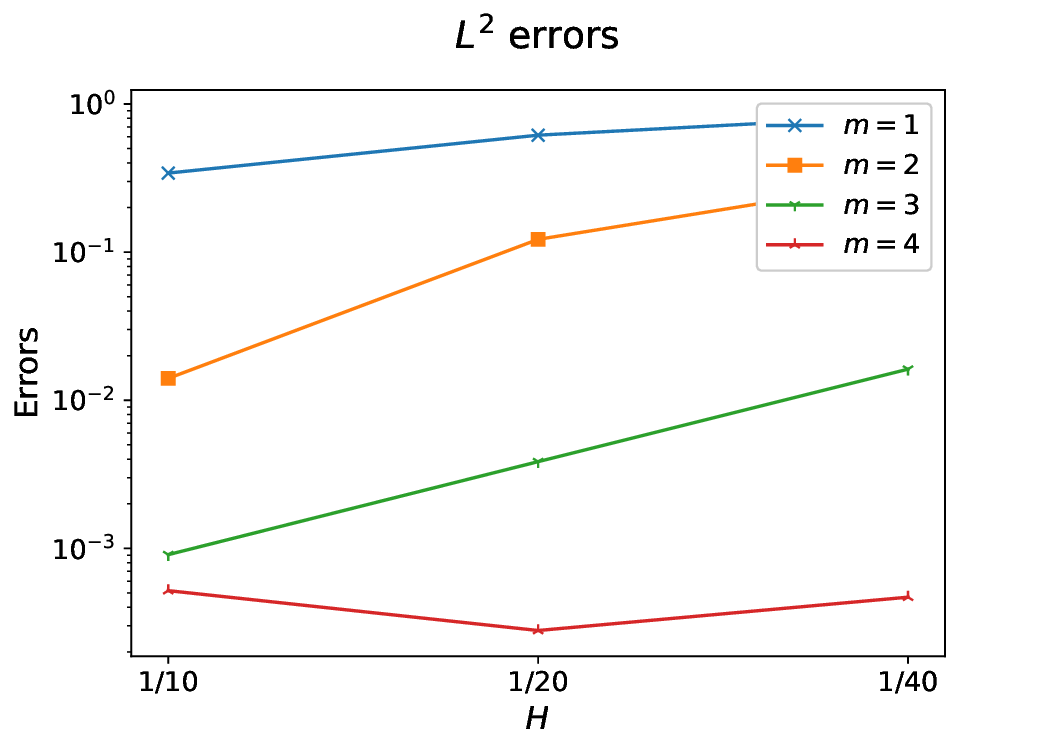}
\includegraphics[scale=0.36]{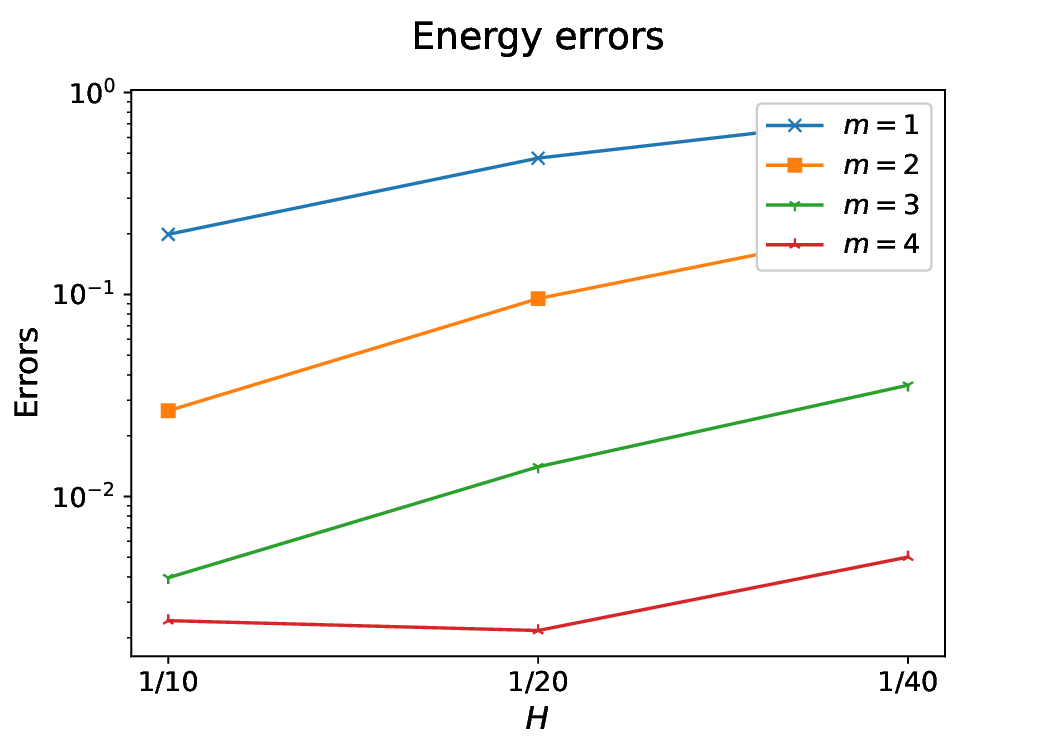}
\caption{Relative L2 error (Left) and Relative h1 error (Right) of Relaxed version of CEM-GMsFEM. In here, we consider case 2, the source term and domain are depicted in \autoref{pic:case2_f}.}
\label{pic:case2_error_r}
\end{figure}

\begin{figure}[htbp]
\centering
\includegraphics[scale=0.33]{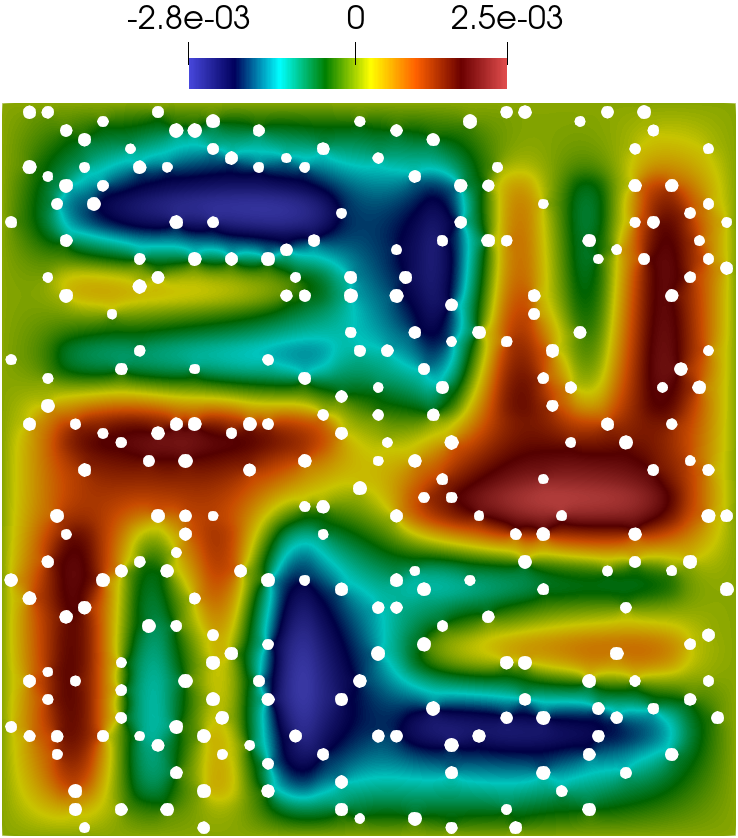}
\includegraphics[scale=0.33]{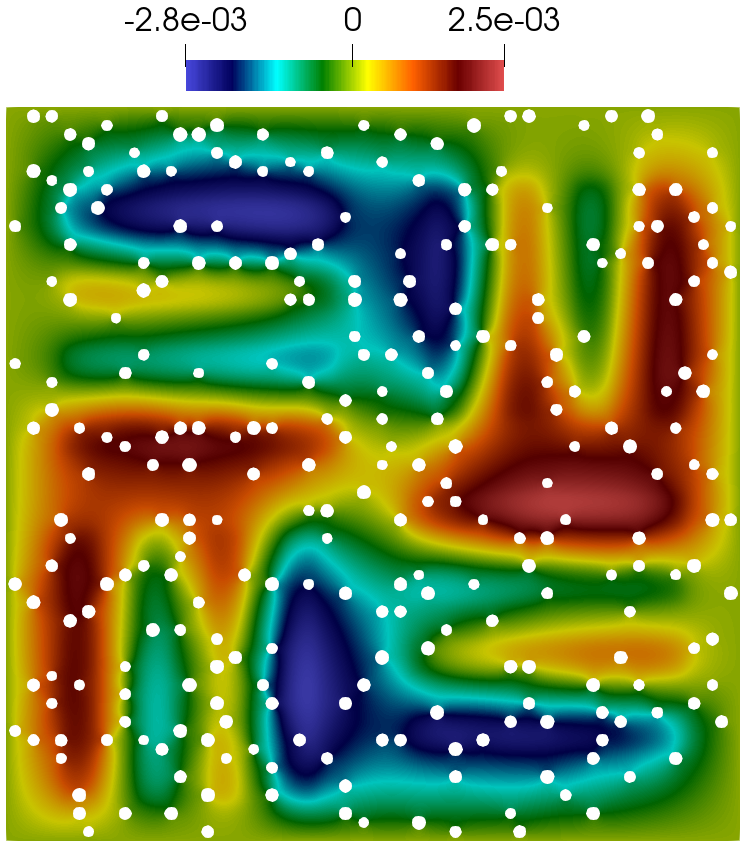}
\includegraphics[scale=0.33]{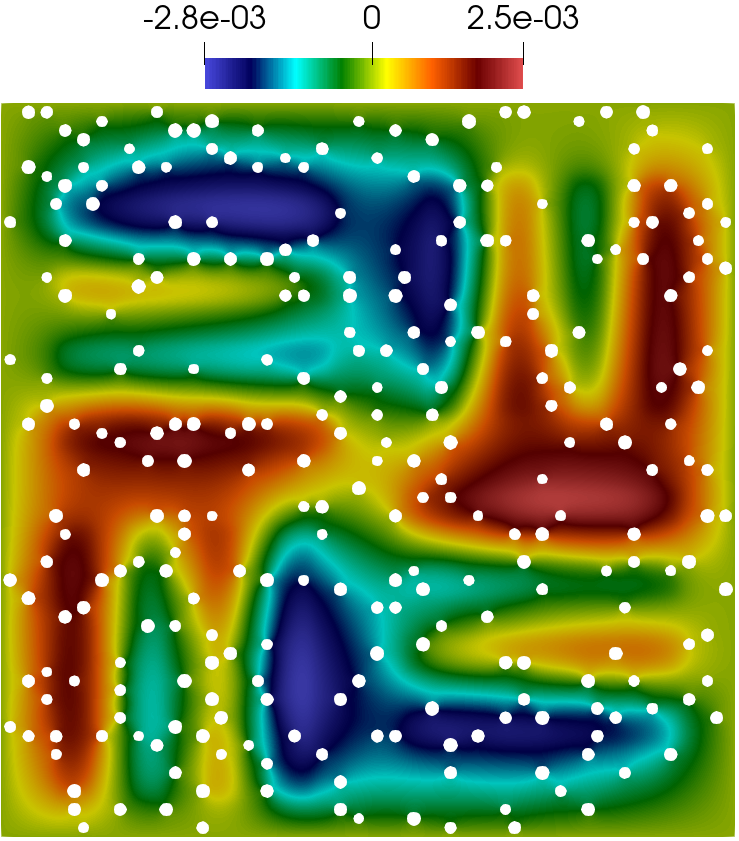}
\caption{Solutions: Reference solution (top left). Using $H=1/40$ and oversampling $3$ coarse layer, constraint version (top right), relaxed version (bottom). In here, we consider case 2, the source term and domain are depicted in \autoref{pic:case2_f}.}
\label{pic:case2_cemuh_c}
\end{figure}

\begin{figure}[htbp]
\centering
\includegraphics[scale=0.34]{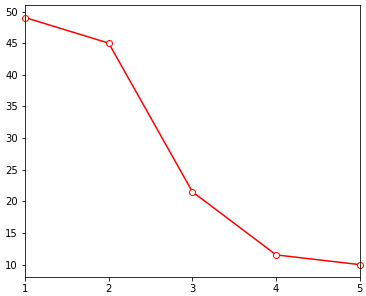}
\hspace{0.7cm}0
\includegraphics[scale=0.27]{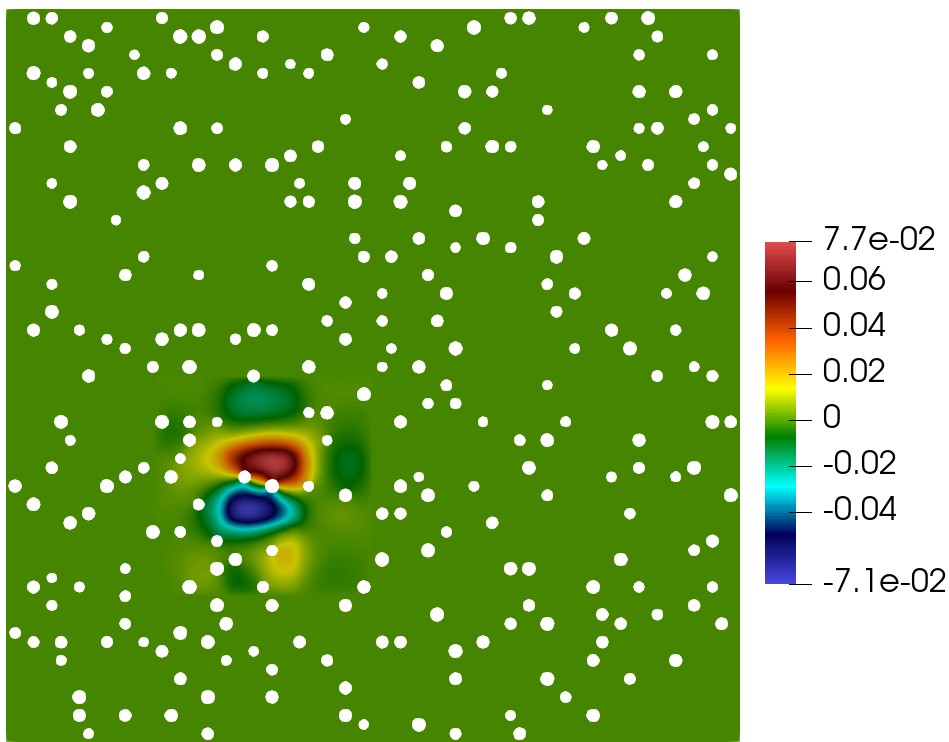}
\includegraphics[scale=0.27]{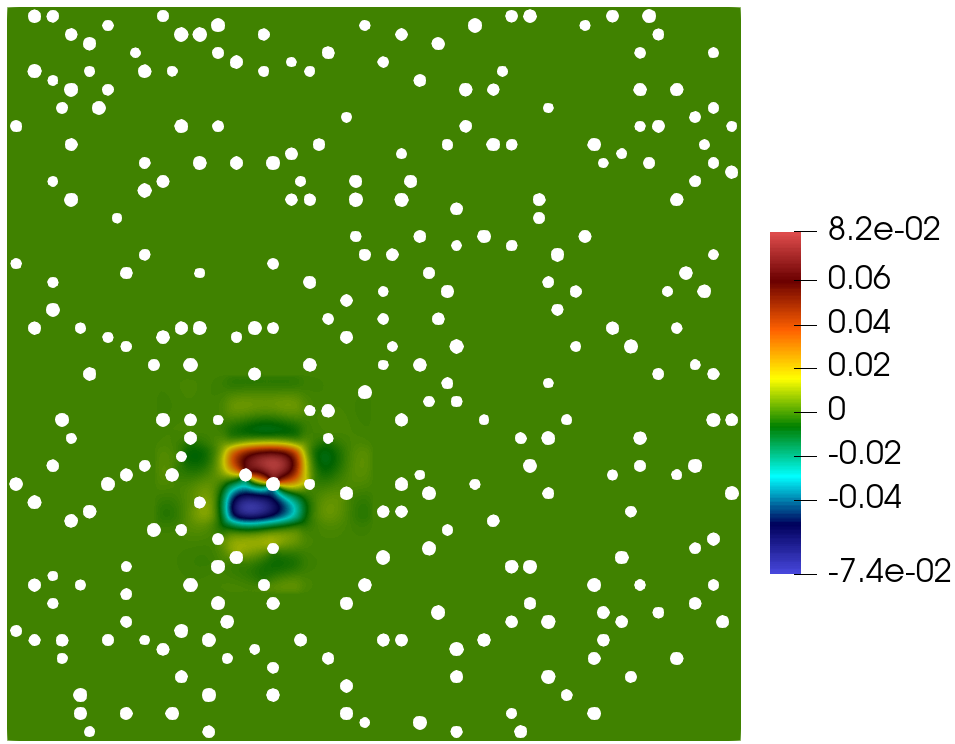}
\caption{A plot of the inverse of first five nonzero eigenvalues (top left), a multiscale basis function using one nonzero eigenfunction in each local auxiliary space (top right), and a multiscale basis functions using five eigenfunctions in each local auxiliary space (bottom).}
\label{pic:case2_K33_decay}
\end{figure}

In \autoref{tab:case2_error}, we outlines the errors corresponding to the refinement of the coarse mesh size with judiciously chosen oversampling layers. 
Both the constraint and relaxed versions of the multiscale solution demonstrate convergence toward the fine-grid solution, as evidenced by diminishing relative L2 and h1 errors with decreasing $H$. 
Further exploration involves varying oversampling layers while maintaining a constant coarse mesh size $H$.
The errors of the CEM-GMsFEM for distinct coarse sizes and oversampling coarse layers are illustrated in \autoref{pic:case2_error_c} (constraint version) and \ref{pic:case2_error_r} (relaxed version). 
These figures distinctly highlight the advantageous linear reduction in error with increasing coarse mesh size inherent in our proposed approach.
To visually underscore these observations, a subset of numerical solutions from \autoref{pic:case2_error_c} and \ref{pic:case2_error_r} is selected. 
The corresponding solutions are depicted in \autoref{pic:case2_cemuh_c}.
In \autoref{pic:case2_K33_decay}, we display the first five nonzero eigenvalues obtained from solving the local spectral problem \eqref{eq:aux_weak}, and a multiscale basis function with one eigenfunction and five eigenfunctions in the local auxiliary space. We can observe that if enough eigenfunctions are exploited in solving the energy minimization problem, then the multiscale basis functions have a fast decay outside of the coarse block.

\section{Conclusions} \label{sec:conclusions}
In this study, we introduce and examine a generalized multiscale finite element method designed to minimize constraint energy for solving the Poisson problem in perforated domains. The method initiates by establishing an auxiliary space that employs eigenvectors associated with small eigenvalues in the local spectral problem. Subsequently, leveraging the principles of constraint energy minimization and oversampling, we generate two distinct multiscale basis functions. Our theoretical analysis suggests that with appropriate selection of the oversampling layer, the resulting multiscale basis functions exhibit a decay property. To validate the effectiveness of our approach, we present two numerical experiments.

\bibliographystyle{plain}
\bibliography{references}
\end{document}